\documentclass[reqno]{amsart}
\makeatletter
\def\oversortoftilde#1{\mathop{\vbox{\m@th\ialign{##\crcr\noalign{\kern3\p@}%
      \sortoftildefill\crcr\noalign{\kern3\p@\nointerlineskip}%
      $\hfil\displaystyle{#1}\hfil$\crcr}}}\limits}

\def\sortoftildefill{$\m@th \setbox\z@\hbox{$\braceld$}%
  \braceld\leaders\vrule \@height\ht\z@ \@depth\z@\hfill\braceru$}

\makeatother

\usepackage{amscd, amsfonts, amsmath,amssymb,amsthm, mathtools} 
\usepackage{tabularx,ltablex}
\usepackage[all,color]{xy}
\usepackage[hidelinks]{hyperref}
\usepackage{cleveref}
\usepackage{float}
\usepackage{changes}
\usepackage{mathrsfs} 
\usepackage{amsmath,latexsym, bbold}
\usepackage{endnotes}
\usepackage{ulem}
\usepackage{aligned-overset}
\usepackage[margin=1in]{geometry}
\usepackage{yhmath}
\usepackage{scalerel}
\usepackage{stackengine,wasysym}

\allowdisplaybreaks
\newtheorem{introques}{Question}
\newtheorem{introdefn}[introques]{Definition}
\newtheorem{introdethm}[introques]{Theorem}
\newtheorem{introgoal}[introques]{Goal}

\newtheorem{thm}{Theorem}[subsection]

\newtheorem{defn}[thm]{Definition}
 
\newtheorem{lemma}[thm]{Lemma} 
\newtheorem{proposition}[thm]{Proposition} 
\newtheorem{remark}[thm]{Remark} 
\newtheorem{Cor}[thm]{Corollary}

\newtheorem{conj}[thm]{Conjecture}
\numberwithin{thm}{subsection}

\DeclareMathOperator{\id}{id}
\DeclareMathOperator{\Gr}{Gr}
\DeclareMathOperator{\hdet}{hdet} 
\DeclareMathOperator{\Ext}{Ext}
\DeclareMathOperator{\Hom}{Hom}
\DeclareMathOperator{\HOM}{HOM}

\DeclareMathOperator{\comod}{comod}
\DeclareMathOperator{\Tor}{Tor}

\DeclareMathOperator{\co}{co}
\DeclareMathOperator{\op}{op}

\newcommand{\uaut}{\underline{\rm aut}}

\newcommand{\era}{\underline{\rm end}^r(A)}

\newcommand{\ara}{\underline{\rm aut}^r(A)}

\newcommand{\kk}{\Bbbk}

\newcommand{\uend}{\underline{\rm end}}

\def\it{\textit}

\numberwithin{equation}{section}

\title{Twisting Manin's universal quantum groups \mbox{and comodule algebras}}

\author[Huang]{Hongdi Huang}

\address{(Huang) Department of Mathematics, Shanghai University, Shanghai 200444, China. and \newline \indent Department of Mathematics, Rice University, Houston, TX 77005, U.S.A.}
\email{h237huan@rice.edu}

\author[Nguyen]{Van C. Nguyen}
\address{(Nguyen) Department of Mathematics, United States Naval Academy, Annapolis, MD 21402, U.S.A.}
\email{vnguyen@usna.edu}

\author[Ure]{Charlotte Ure}
\address{(Ure) Department of Mathematics, Illinois State University, Normal, IL 61790, U.S.A.}
\email{cure@ilstu.edu}

\author[Vashaw]{Kent B. Vashaw}
\address{(Vashaw) Department of Mathematics,
Massachusetts Institute of Technology,
Cambridge, MA 02139, U.S.A.}
\email{kentv@mit.edu}

\author[Veerapen]{Padmini Veerapen}
\address{(Veerapen) Department of Mathematics, Tennessee Tech University, Cookeville, TN 38505, U.S.A.}
\email{pveerapen@tntech.edu}

\author[Wang]{Xingting Wang}
\address{(Wang) Department of Mathematics, Louisiana State University, Baton Rouge, Louisiana 70803, USA}
\email{xingtingwang@math.lsu.edu}

\date\today

\subjclass{
16T05, 
16W50, 
17B37 
}
\keywords{universal quantum group, Morita--Takeuchi equivalence, 2-cocycle twist, Artin--Schelter regular algebra, superpotential algebra}

\begin{document}

\maketitle

\begin{abstract} 
We introduce the notion of quantum-symmetric equivalence of two connected graded algebras, based on Morita--Takeuchi equivalences of their universal quantum groups, in the sense of Manin. We study homological and algebraic invariants of quantum-symmetric equivalence classes, and prove that numerical $\mathrm{Tor}$-regularity, Castelnuovo--Mumford regularity, Artin--Schelter regularity, and the Frobenius property are invariant under any Morita--Takeuchi equivalence. In particular, by combining our results with the work of Raedschelders and Van den Bergh, we prove that Koszul Artin--Schelter regular algebras of a fixed global dimension form a single quantum-symmetric equivalence class. Moreover, we characterize 2-cocycle twists (which arise as a special case of quantum-symmetric equivalence) of Koszul duals, of superpotentials, of superpotential algebras, of Nakayama automorphisms of twisted Frobenius algebras, and of Artin--Schelter regular algebras. We also show that finite generation of Hochschild cohomology rings is preserved under certain 2-cocycle twists.
\end{abstract}

\section{Introduction}

Symmetry is an important concept that appears in mathematics and theoretical physics. While classical symmetries arise from group actions on polynomial rings, quantum symmetries are introduced to understand certain quantum objects which appear in the theory of quantum mechanics and quantum field theory. Examples of such quantum objects include Hopf algebras, subfactors and topological phases of matter, whose symmetries are described by tensor categories, resembling categories of group representations. 

In this paper, we focus on algebraic quantum symmetries of noncommutative projective spaces. In noncommutative projective algebraic geometry, noncommutative projective spaces are described by their noncommutative homogeneous coordinate rings, among which is the family of Artin--Schelter (AS) regular algebras. In recent years, finite-dimensional quantum symmetries have been extensively studied in connection with AS-regular algebras, for example, in the context of noncommutative invariant theory, rigidity, and the McKay correspondence (see e.g.,~\cite{CKWZ, ChanKirkmanWaltonZhang2019, CS1,CS2,CuadraEtingofWalton2016, EtingofWalton2016, EtingofWalton2017, M2005}). 

We propose a systematic study of infinite-dimensional quantum symmetries for noncommutative projective spaces. By following Manin's method in \cite{Manin2018}, we describe these quantum symmetries in terms of the universal coaction of a particular infinite-dimensional Hopf algebra, called Manin's universal quantum group (see \Cref{def:ManinU}), on the corresponding noncommutative homogeneous coordinate rings while preserving the grading. This motivates the following categorical notion that describes when two noncommutative projective spaces possess the same quantum symmetries.  

\begin{introdefn}
\label{DefnA}
Let $A$ and $B$ be two connected graded algebras finitely generated in degree one. We say $A$ and $B$ are \emph{weakly quantum-symmetrically equivalent} if there is a monoidal equivalence between the comodule categories of their associated universal quantum groups 
\[
\comod(\underline{\rm aut}(A))~\overset{\otimes}{\cong}~{\rm comod}(\underline{\rm aut}(B))
\]
in the sense of Manin. If this equivalence additionally sends  $A$ to $B$ as comodule algebras, we say that $A$ and $B$ are \emph{quantum-symmetrically equivalent}.
\end{introdefn}
When $A$ is nonconnected graded, interested readers can refer to \cite{HWWW} for a description of universal semigroupoids coacting on $A$. When $A$ and $B$ are two Koszul AS-regular algebras of the same dimension, by a celebrated result of Raedschelders and Van den Bergh \cite{vdb2017}, $A$ and $B$ are always weakly quantum-symmetrically equivalent. Thus, Koszul AS-regular algebras can be viewed as deformations of the projective spaces $\mathbb P^{n}$ that preserve their quantum symmetries. This motivates the following main goal.

\begin{introgoal}
\label{QuestionB}
Describe the (weak) quantum-symmetric equivalence class of a connected graded algebra $A$ finitely generated in degree one.
\end{introgoal}

In our first result, we give a way of deforming an algebra $A$ which preserves its quantum-symmetric equivalence class, by utilizing 2-cocycle twists of Hopf algebras and their comodule categories. The notion of a 2-cocycle twist was introduced by Doi and Takeuchi \cite{Doi93,DT94} as a dual version of the Drinfeld twist. A 2-cocycle twist of a Hopf algebra $H$ deforms its algebra structure by a $2$-cocycle $\sigma$ on $H$ but preserves its coalgebra structure. It is well-known that $H$ and its 2-cocycle twist $H^\sigma$ are Morita--Takeuchi equivalent, that is, their corresponding comodule categories are monoidally equivalent. By applying this monoidal equivalence of comodule categories, we establish the following result. 

\begin{introdethm}[\Cref{thm:action-Manin}]\label{Intro:Zhang2}
Let $A$ be any $\mathbb Z$-graded locally finite algebra and $H$ be any Hopf algebra that right coacts on $A$ preserving its grading. Then for any right 2-cocycle $\mu$ on $ H$, there is a right 2-cocycle $\sigma$ on Manin's universal quantum group $\underline{\rm aut}^r(A)$ associated to $A$ such that there is an isomorphism between 2-cocycle twists 
\[A_{\sigma}~\cong~A_{\mu}\]
as graded algebras, and 
\[\underline{\rm aut}^r(A_\mu)~\cong~ \underline{\rm  aut}^r(A)^\sigma\]
as Hopf algebras. As a consequence, $A_\mu$ and $A$ are quantum-symmetrically equivalent. 
\end{introdethm}

Another approach of deforming a graded algebra $A$ was observed by Artin, Tate, and Van den Bergh in \cite{ATV1991} and Zhang in \cite{Zhang1996} as a twist of the original graded product of $A$ by a graded automorphism $\phi$ of $A$, or more generally by a twisting system. Two graded algebras have equivalent graded module categories if and only if they are twists of each other in the sense of Zhang \cite[Theorem 3.5]{Zhang1996}. We will refer to this deformation by the term Zhang twist and denote it by $A^\phi$. In \cite{HNUVVW21}, the notion of twisting pairs was introduced to connect Zhang twists and 2-cocycle twists for a graded Hopf algebra satisfying some twisting conditions. An alternate approach to the connection between Zhang twists and 2-cocycle twists was also given by Bichon, Neshveyev, and Yamashita in \cite[Theorem 2.8 and Remark 2.9]{Bichon-Neshveyev-Yamashita2016} and \cite[Remark 2.4]{Bichon-Neshveyev-Yamashita2018}. In \cite[Theorems E and F]{HNUVVW21}, it was shown that Zhang twists of quadratic algebras yield Morita--Takeuchi equivalent universal quantum groups, or in light of \Cref{DefnA}, Zhang twists preserve weak quantum-symmetric equivalence classes. In our next result, for a graded algebra $A$, we establish a one-to-one correspondence between graded automorphisms of $A$ and twisting pairs on $\underline{\rm aut}^r(A)$. By viewing $A$ as a comodule algebra over Manin's universal quantum group $\underline{\rm aut}^r(A)$, we show that a Zhang twist of $A$ can be realized as a 2-cocycle twist, proving that the weak quantum-symmetric equivalence of $A$ and $A^\phi$ is in fact a quantum-symmetric equivalence.

\begin{introdethm}[\Cref{lem:2cocycleGrade}]\label{Introthm:2cycleZhang}
Let $A$ be a connected graded algebra finitely generated in degree one subject to $m$-homogeneous relations. The following groups are isomorphic:
\begin{itemize}
    \item[(1)] the group of graded automorphisms of $A$;
    \item[(2)] the group of twisting pairs of $\underline{\rm aut}^r(A)$ under component-wise composition; 
    \item[(3)] the group of one-dimensional representations of $\underline{\rm aut}^r(A)$ under tensor product of $\underline{\rm aut}^r(A)$-modules. 
\end{itemize}
Moreover, for any graded automorphism $\phi$ of $A$, let $\sigma$ be the right 2-cocycle on $\uaut^r(A)$ given by the twisting pair associated to $\phi$. Then \[A^{\phi}\cong A_{\sigma}\] as graded algebras. As a consequence, $A^\phi$ and $A$ are quantum-symmetrically equivalent.
\end{introdethm}

Motivated by Manin's perspective \cite[Introduction]{Manin2018}, progress has been made recently in an attempt to answer whether the universal quantum groups that coact on a graded algebra $A$ possess the same homological and ring-theoretic properties as $A$, see e.g., \cite{Chirvasitu-Walton-Wang2019, vdb2017, WaltonWang2016}. This connection has been investigated for AS-regular algebras of dimension two in \cite{WaltonWang2016} and for arbitrary $N$-Koszul AS-regular algebras in \cite{Chirvasitu-Walton-Wang2019}. In light of \Cref{QuestionB}, we provide partial answers to the following question.

\begin{introques}
\label{homol-invt-qse}
Which ring-theoretical and homological properties of connected graded algebras are invariant under quantum-symmetric equivalence (in the sense of \Cref{DefnA})?  
\end{introques}

Suppose $A$ is a connected graded algebra, $H$ is a Hopf algebra coacting on $A$, and $F: \comod(H) \to \comod(H')$ is a Morita--Takeuchi equivalence, for another Hopf algebra $H'$. In this case, $F(A)$ is an $H'$-comodule algebra. It is natural to additionally ask the following generalization of \Cref{homol-invt-qse}: which ring-theoretical and homological properties of $A$ are shared by $F(A)$? For the sake of brevity, when we refer to a property of $A$ being invariant under Morita--Takeuchi equivalence, we mean that this property is shared by $F(A)$, for any $F, H,$ and $H'$ as above.

Indeed, as a special case of Morita--Takeuchi equivalence, 2-cocycle twists under finite-dimensional Hopf coactions have been studied earlier by various authors (see e.g., \cite{CKWZ, CS1, CS2,M2005}). In our next result, we prove that several homological properties for connected graded comodule algebras, including certain numerical regularities as discussed by Dong and Wu in \cite{DongWu2009}, J\o rgensen in \cite{Jorgensen1999,Jorgensen2004}, and Kirkman, Won, and Zhang in \cite{Kirkman-Won-Zhang2021} (see \Cref{def:Torreg} for a precise definition), are preserved under Morita--Takeuchi equivalences of the Hopf algebras that coact on them. In each case, we show that the given property is determined by a certain category of relative modules, in other words, a category of modules internal to a category of comodules (see \cite[Section 7.8]{EGNO}).

\begin{introdethm}[\Cref{thm:Torreg}]\label{Intro:numregular}
Let $H$ be any Hopf algebra with a bijective antipode and $A$ be a noetherian connected graded $H$-comodule algebra. The following numerical regularities are invariant under any Morita--Takeuchi equivalence:
 \begin{itemize}
 \item[(1)] ${\rm Torreg}(\kk)$, where $\kk$ is considered as a trivial module over $A$;
 \item[(2)] ${\rm CMreg}(A)$; and
 \item[(3)] ${\rm ASreg}(A)$.
\end{itemize}
Additionally, if $A$ is either
\begin{itemize}
 \item[(4)] $s$-Cohen--Macaulay, or 
 \item[(5)] AS-regular of dimension $d$,
 \end{itemize}
then so is its image under Morita--Takeuchi equivalence. In particular, (1)-(5) are quantum-symmetric invariants.
\end{introdethm}

We apply the general theory developed above to the specific case of 2-cocycle twists of comodule algebras. Provided that the underlying comodule algebra $A$ is connected graded, in view of our \Cref{Introthm:2cycleZhang}, when we realize any Zhang twist of $A$ as a 2-cocycle twist, our results in \Cref{Intro:numregular} generalize the well-known ones in the context of Zhang twist. 

For the first application, we define a 2-cocycle twist of a superpotential, and examine twists of superpotential algebras and their related universal quantum groups. We are interested in superpotential algebras due to the fact that every $N$-Koszul AS-regular algebra can be concretely realized as a superpotential algebra \cite[Theorem 4.3]{Dubois-Violette2005}; there are further results that also connect superpotential algebras to graded Calabi--Yau algebras, see e.g.,  \cite{BCY3,BSW, DVM, VdBSup}. In \cite{Mori-Smith2016}, Mori and Smith investigated twisting of superpotentials in relation to Zhang twists of their associated superpotential algebras. We extend their result to twisting by an arbitrary 2-cocycle of a Hopf algebra $H$ that coacts on the superpotential algebra while preserving its grading. 

\begin{introdethm}[\Cref{prop:2.2.5}]\label{Intro:super}
Let $1\leq N\leq m$ be integers and $H$ be a Hopf algebra over a base field $\kk$. Let $V$ be a finite-dimensional right $H$-comodule, $j:W \hookrightarrow V^{\otimes m}$ be an $H$-subcomodule and $\sigma$ be a left 2-cocycle on $H$. Then there is an isomorphism \[A(W_{\sigma},N)~\cong~A(W,N)_{\sigma}\] of $H^{\sigma}$-comodule algebras, where $A(W,N)$ is the derivation-quotient algebra given in \Cref{defn:derivation}, and $W_\sigma$ is from \Cref{def:twistform}.
\end{introdethm}

We now focus on 2-cocycle twists of AS-regular algebras, when they are considered as comodule algebras over a Hopf algebra $H$. For the second application, we examine whether the AS-regular property is preserved under 2-cocycle twists. When $H$ is semisimple, this was proved by Chan, Kirkman, Walton, and Zhang \cite{CKWZ}. Recently, this was proved when $H$ is finite-dimensional by Davies \cite{Davies2017}, and independently by Chirvasitu and Smith \cite{CS2}. In the following result, we remove the assumption that $H$ is finite-dimensional.

\begin{introdethm}[\Cref{thm:AS}, \Cref{thm:TwistAS}]\label{Intro:AStwist} 
Let $A$ be an $N$-Koszul AS-regular algebra of dimension $d$.
\begin{enumerate}
    \item Suppose $H$ is an arbitrary Hopf algebra that right coacts on $A$ preserving its grading. Then for any right 2-cocycle $\sigma$ on $H$, the twisted algebra $A_\sigma$ is again $N$-Koszul AS-regular of the same dimension. 
    \item Suppose $\kk$ is algebraically closed and $N=2$. Let $B$ be any connected graded algebra generated in degree one. Then $A$ and $B$ are quantum-symmetrically equivalent if and only if $B$ is  a Koszul AS-regular algebra of the same global dimension as $A$. 
    \item With the same assumptions as in (2), $A$ and $B$ are $2$-cocycle twists of each other if and only if they have the same Hilbert series.
\end{enumerate}
\end{introdethm}
As an immediate consequence, we deduce in \Cref{Cor:qs-equivalence} that the quantum-symmetric equivalence class of the polynomial algebra $\kk[x_1,\ldots,x_d]$ consists of all Koszul AS-regular algebras of dimension $d$.    

In light of this result, we can bring classical questions on AS-regular algebras (for instance, whether AS-regular algebras are noetherian, or domains \cite[Remark on p.~ 338]{ATV1991}, \cite[Question 2.1.8]{Ro2016}) into the context of 2-cocycle twists. In \cite[Theorem 7.2.3]{vdb2017}, Raedschelders and Van den Bergh proved that Koszul AS-regular algebras of the same dimension are always quantum-symmetrically equivalent. By \Cref{Intro:AStwist}, this result is strengthened to an if-and-only-if statement, and the Koszul AS-regular property can be completely characterized at the level of universal quantum groups.

For connected graded algebras, the AS-regular property is equivalent to the twisted Calabi--Yau property \cite[Lemma 1.2]{RRZ1}. To be more precise, a graded twisted Calabi--Yau algebra differs from a graded Calabi--Yau algebra by a unique graded automorphism called the \emph{Nakayama automorphism} derived from the twisted bimodule structure on its dualizing module. Many researchers have calculated the Nakayama automorphisms of various families of twisted Calabi--Yau algebras, see e.g., \cite{BZ08, goodearl2015unipotent, liu2014twisted, RRZ1, RRZ2, yekutieli2000rigid, yu2023calabi}. In general, the Nakayama
automorphism is a subtle, but important, invariant and has applications in noncommutative invariant theory and the Zariski cancellation problem for noncommutative algebras \cite{lu2017nakayama}. For the third application, we describe the Nakayama automorphisms of $N$-Koszul AS-regular algebras under arbitrary 2-cocycle twists.

\begin{introdethm}
[\Cref{Cor:Naka}]\label{Intro:Naka}
Let $A$ be an $N$-Koszul AS-regular algebra and $H$ be any Hopf algebra that right coacts on $A$. Then for any right 2-cocycle $\sigma$ on $H$, there is a formula of the Nakayama automorphism of the twisted AS-regular algebra $A_\sigma$, explicitly described in terms of the Nakayama automorphism of $A$, the $H$-coaction on $A$, and its homological codeterminant.  
\end{introdethm}

Finally, for the fourth application, we examine the finite generation of the Hochschild cohomology ring under some 2-cocycle twists of the underlying algebra. This finite generation is an essential assumption to the theory of support varieties, which originates from group representation theory and enables one to use geometric approaches to study representations \cite{Witherspoon2019}. Our result provides more examples of algebras whose Hochschild cohomology rings are finitely generated.

\begin{introdethm}
[\Cref{thm:FGCHC}]\label{Intro:HH}
Let $H$ be a finite-dimensional commutative Hopf algebra, $A$ be a right $H$-comodule algebra, and $\sigma$ be any right 2-cocycle on $H$. If the Hochschild cohomology ring of $A$ is finitely generated, then so is that of $A_\sigma$.
\end{introdethm}

\subsection*{Acknowledgements} 

The authors thank Ellen Kirkman and James Zhang for useful discussions, and thank the referee for their careful reading and suggestions which improved the paper. Some results in this paper were formulated at the Structured Quartet Research Ensembles (SQuaREs) program in March 2022, and at the BIRS Workshop on ``Noncommutative Geometry and Noncommutative Invariant Theory" in September 2022. The authors thank the American Institute of Mathematics, the Banff International Research Station, and the organizers of the BIRS Workshop for their hospitality and support. Nguyen was partially supported by the Naval Academy Research Council and NSF grant DMS-2201146. Ure was partially supported by an AMS--Simons Travel Grant. Veerapen was partially supported by an AWM--NSF Travel Grant. Vashaw was partially supported by NSF Postdoctoral Fellowship DMS-2103272. Wang was partially supported by Simons collaboration grant \#688403 and AFOSR grant FA9550-22-1-0272.

\subsection*{Conventions} 

Throughout the paper, let $\kk$ be a base field with $\otimes$ taken over $\kk$ unless stated otherwise. All categories are $\kk$-linear and all algebras are associative over $\kk$.  A $\mathbb Z$-graded algebra $A=\bigoplus_{i\in \mathbb Z} A_i$ is called \emph{connected graded} if $A_i=0$ for $i<0$ and $A_0=\kk$. A $\mathbb{Z}$-graded algebra $A$ is called \emph{locally finite} if $\dim_{\kk} A_i < \infty$ for all $i$. We use the Sweedler notation for the coproduct in a coalgebra $B$: for any $h \in B$, $\Delta(h) = \sum h_1 \otimes h_2 \in B \otimes B$. When a bialgebra $B$ right (resp.~ left) coacts on an algebra $A$, we denote the \emph{right} coaction $\rho: A \to A \otimes B$ by $a \mapsto \sum a_0 \otimes a_1 $ (resp.~ the \emph{left} coaction $\rho: A \to B \otimes A$ by $a \mapsto \sum a_{-1} \otimes a_0$). When we refer to a Hopf algebra $H$ coacting on an algebra $A$, we mean that $A$ is an $H$-comodule algebra.
The category of all (resp.~ finite-dimensional) right $B$-comodules is denoted by ${\rm comod}(B)$ (resp.~${\rm comod}_{\rm fd}(B))$.

\section{Twisting Manin's universal quantum groups}
\label{sec:TwistingManin}
 
In this section, we study a family of infinite-dimensional Hopf algebras, called Manin's universal quantum groups, denoted by $\underline{{\rm aut}}(A)$, which universally coact on connected graded quadratic algebras $A$ \cite{Manin2018}. Here, we consider $\underline{\rm aut}(A)$ under a more general setting, namely, when $A$ is finitely generated subject to $m$-homogeneous relations. For such algebras $A$, we study the 2-cocycle twist of $\underline{{\rm aut}}(A)$ and show that a Zhang twist of $A$ can be realized as a 2-cocycle twist, showing that quantum-symmetric equivalence is a graded Morita invariant.

\subsection{Manin's universal quantum groups}\label{subsec:Manin twists}

Let $m\ge 2$ be an integer. In this subsection, we give an explicit construction of Manin's universal quantum group associated to a graded connected algebra $A$, which is finitely generated subject to $m$-homogeneous relations. We recall Manin's definition here \cite{Manin2018}.

\begin{defn}
\label{def:ManinU} 
Let $A$ be any $\mathbb Z$-graded locally finite $\kk$-algebra. The \emph{right universal bialgebra $\underline{\rm end}^r(A)$ associated to $A$} is the bialgebra that right coacts on $A$ preserving the grading of $A$ via $\rho: A\to A \otimes \underline{\rm end}^r(A)$ satisfying the following universal property: if $B$ is any bialgebra that right coacts on $A$ preserving the grading of $A$ via $\tau: A\to A \otimes B$, then there is a unique bialgebra map $f: \underline{\rm end}^r (A)\to B$ such that the following diagram 
\begin{align}
\label{def:aut}
\xymatrix{
A\ar[r]^-{\rho}\ar[dr]_-{\tau} & A \otimes \underline{\rm end}^r(A) \ar[d]^-{\id \otimes f} \\
& A \otimes B
}
\end{align}
commutes. The right universal quantum group $\uaut^r(A)$ is the universal Hopf algebra right coacting on $A$. The universal left-coacting bialgebra and Hopf algebra, $\uend^l(A)$ and $\uaut^l(A)$, respectively, are defined analogously.
\end{defn} 

\begin{remark}
Throughout the paper, statements referring to $\uaut(A)$ will mean that the analogous statements for $\uaut^r(A)$ and $\uaut^l(A)$ are both true (and likewise for $\uend(A)$). 
\end{remark}

While universal quantum groups may or may not exist for an arbitrary graded algebra $A$, the following result guarantees its existence when $A$ is locally finite $\mathbb Z$-graded.

\begin{thm}\cite[Theorem 3.16 and Example 4.8(1)-(2)]{AGV}
If $A$ is a locally finite $\mathbb Z$-graded algebra, then $\uaut(A)$ exists.    
\end{thm}

In what follows, we provide a concrete way to construct $\uaut(A)$ if $A$ is a connected graded algebra generated by $x_1,\ldots,x_n$ in degree one, subject to $m$-homogeneous relations for some integer $m>2$. This extends the explicit construction in \cite[Chapter 6]{Manin2018} and \cite[\S 2.1]{CWZ2014}. Suppose that the $m$-homogeneous relations are given by $\{r_w=0 \mid 1\leq w \leq t\}$, where
\[
r_w:=\sum_{1\leq i_1,\ldots,i_m\leq n} c_w^{i_1\cdots i_m} x_{i_1}\cdots x_{i_m}, 
\] 
with $c_w^{i_1\cdots i_m}\in \kk$. We assume that $\{r_1,\ldots,r_t\}$ are linearly independent. We write $A=\kk\langle A_1\rangle/(R)$, where $A_1={\rm span}_\kk\{x_1,\ldots,x_n\}$ and $R=\bigoplus_{i=1}^t\kk r_i\subseteq (A_1)^{\otimes m}$. The \emph{$m$-Koszul dual} of $A$, denoted by $A^!$, is an algebra $\kk\langle A_1^*\rangle/(R^\perp)$, where $A_1^*=\Hom(A_1,\kk)$ and 
\begin{equation}\label{eq:Rperp}
R^\perp:=\left\{a\in (A_1^*)^{\otimes m} \mid \langle a,R\rangle =0\right\}=\left\{a\in (A_1^*)^{\otimes m} \mid \langle a,r_w\rangle=0 \text{ for all } 1\leq w\leq t \right\}.
\end{equation}
Let $\{x^1,\ldots,x^n\}$ be the dual basis for $(A^!)_1= A_1^*$ relative to the standard basis $\{x_1, \ldots, x_n\}$ of $A_1$ and choose a basis for $R^\perp$, say $\{r^u\}$, where we write 
\[
r^u=\sum_{1\leq i_1,\ldots,i_m\leq n} c^u_{i_1\cdots i_m} x^{i_1}\cdots x^{i_m}
\]
for all $1\leq u\leq n^m-t$. According to \eqref{eq:Rperp}, we have
\begin{equation}\label{eq:mrelation}
    \sum_{1\leq i_1,\ldots,i_m\leq n} c^u_{i_1\cdots i_m}c_w^{i_1\cdots i_m}=0
\end{equation}
for all possible $u,w$. 
The following lemma is used to induce a coaction on $A$ from a coaction on the free algebra $\kk\langle x_1,..., x_n \rangle$, which extends the well-known result in the quadratic case \cite[Lemma 2.2]{CWZ2014} to any connected graded algebra with $m$-homogeneous relations.

\begin{lemma}
\label{lemma:bialgebraM}
Let $A$ be the quotient algebra of the free algebra $A'=\kk\langle x_1,\dots,x_n\rangle$ by $m$-homogeneous relations $R=\bigoplus_{w=1}^t\kk\, r_w$ as above. Let $B$ be a bialgebra that right coacts on $A'$ with $\rho(x_i)=\sum_{j=1}^n x_j\otimes b_{ji}$ for some $b_{ji}\in B$. The coaction $\rho: A'\to A'\otimes B$ satisfies $\rho(R)\subseteq R\otimes B$, if and only if
\begin{equation}\label{eq:coactionR}
\sum_{\substack{1\leq i_1,\ldots,j_m\leq n \\ 1\leq j_1,\ldots,j_m\leq n}} c_w^{i_1\cdots i_m}c^u_{j_1\cdots j_m}b_{j_1i_1}\cdots b_{j_mi_m}=0
\end{equation}
for all possible $u,w$. 
\end{lemma}

\begin{proof}
Since $R\subseteq (A_1)^{\otimes m}$, we have $\rho(R)\subseteq (A_1)^{\otimes m} \otimes B$. By duality, 
\[
R=\left\{r\in (A_1)^{\otimes m} \mid \langle R^\perp,r\rangle=0\right\}.
\] 
Thus, $\rho(R)\subseteq R\otimes B$ if and only if $\langle R^\perp \otimes \id_B,\rho(R)\rangle=0$ if and only if \eqref{eq:coactionR} holds.
\end{proof}

For two connected graded algebras $A=\kk\langle A_1\rangle/(R(A))$ and $B=\kk\langle B_1\rangle/(R(B))$ with $m$-homogeneous relations $R(A)\subseteq (A_1)^{\otimes m}$ and $R(B)\subseteq (B_1)^{\otimes m}$ respectively, we extend Manin's bullet product \cite[\S 4.2]{Manin2018} to $A$ and $B$ such that
\[
A\bullet B~:=~\frac{\kk \langle A_1\otimes B_1\rangle}{\left(\tau(R(A)\otimes R(B))\right)},
\]
where $\tau: (A_1)^{\otimes m} \otimes (B_1)^{\otimes m}\to (A_1\otimes B_1)^{\otimes m}$ is the permutation considered by shuffling $m$-tensor factors of $B_1$ into $m$-tensor factors of $A_1$, resulting in alternating tensor factors from $A_1$ and $B_1$. 
When $B=A^!=\kk\langle (A_1)^*\rangle/(R(A)^\perp)$ is the $m$-Koszul dual algebra of $A$, we see that $A\bullet A^!$ is a connected graded bialgebra with matrix comultiplication defined below on the generators of $A_1\otimes A_1^*$. As before, choose a basis $\{x_1,\ldots,x_n\}$ for $A_1$ and let $\{x^1,\ldots,x^n\}$ be the dual basis for $(A^!)_1= A_1^*$. Write $z_j^k=x_j\otimes x^k \in A_1\otimes A_1^*$ as the generators for $A\bullet A^!$. Then the coalgebra structure of $A\bullet A^!$ is given by
\[
\Delta(z_j^k)=\sum_{1\leq i\leq n} z_i^k\otimes z_j^i, \quad \text{and} \quad \varepsilon(z_j^k)=\delta_{j,k},\quad \text{for any }  1 \leq j,k \leq n.
\]

The following result is a straightforward generalization of the quadratic case \cite[Proposition 3.1.2, Proposition 3.1.3, Lemma 3.2.1]{HNUVVW21}, which describes Manin's universal bialgebra  $\underline{\rm end}(A)$ and Manin's universal quantum group $\underline{\rm aut}(A)$.

\begin{lemma}
\label{lem:ManinM}
Let $A=\kk\langle A_1\rangle/(R(A))$ and $B=\kk\langle B_1\rangle/(R(B))$ be two connected graded algebras finitely generated in degree one with $m$-homogeneous relations. Then we have:
\begin{itemize}
    \item[(1)] $\underline{\rm end}^r(A)\cong A\bullet A^!$ and $\underline{\rm end}^l(A)\cong A^!\bullet A$; 
    \item[(2)] $\underline{\rm aut}^r(A)$ is the Hopf envelope of $\underline{\rm end}^r(A)$ and $\underline{\rm aut}^l(A)$ is the Hopf envelope of $\underline{\rm end}^l(A)$;
     \item[(3)] $\underline{\rm aut}^r(A)\cong \underline{\rm aut}^l(A^!)$;
    \item[(4)] For any algebra map $f: A\to A\otimes B$ satisfying $f(A_1)\subseteq A_1\otimes B_1$, there exists a unique graded algebra map $g: \underline{\rm end}^r(A)\to B$ such that $f=(\id \otimes g)\circ \rho_A$ where $\rho_A: A\to A\otimes \underline{\rm end}^r(A)$ is the universal bialgebra coaction of $\underline{\rm end}^r(A)$ on $A$;  
    \item[(5)] For any algebra map $f: A\to B\otimes A$ satisfying $f(A_1)\subseteq B_1\otimes A_1$, there exists a unique graded algebra map $g: \underline{\rm end}^l(A)\to B$ such that $f=(g\otimes \id)\circ \rho_{A}$ where $\rho_{A}: A\to \underline{\rm end}^l(A)\otimes A$ is the universal bialgebra coaction of $\underline{\rm end}^l(A)$ on $A$. 
\end{itemize}
\end{lemma}

\begin{proof}
(1): We use the above notation to denote $A=\kk\langle A_1\rangle/(R)$, where $R=\bigoplus_{w=1}^t\kk\, r_w\in (A_1)^{\otimes m}$. Its $m$-Koszul dual algebra is  $A^!=\kk\langle A_1^*\rangle/(R^\perp)$ with $R^\perp =\bigoplus_{u=1}^{n^m-t}\kk\, r^u\in (A_1^*)^{\otimes m}$ satisfying $\langle R^\perp, R\rangle=0$. Hence by definition, we have 
\[A\bullet A^!=\frac{\kk \langle z^k_j\rangle}{\left(\sum_{\substack{1\leq i_1,\ldots,i_m\leq n\\ 1\leq j_1,\ldots,j_m\leq n}} c_w^{i_1\cdots i_m}c^u_{j_1\cdots j_m} z_{i_1}^{j_1}\cdots z_{i_m}^{j_m}\right)}.\]
By \Cref{lemma:bialgebraM}, the bialgebra $A\bullet A^!$ right coacts on $A$ via $\rho: A\to A\otimes (A\bullet A^!)$ by
\[
\rho(x_i)=\sum_{1\leq j\leq n} x_j\otimes z^j_i, \quad \text{ for any } 1 \leq i \leq n.
\]
Suppose $D$ is another bialgebra that right coacts on $A$ via $\rho_D(x_i)=\sum x_j\otimes d_{ji}$ for some $d_{ji}\in D$. It is clear that $\Delta(d_{ji})=\sum d_{jk}\otimes d_{ki}$ and $\varepsilon(d_{ji})=\delta_{ji}$. Since the $d_{ji}$'s satisfy the relations \eqref{eq:coactionR} by \Cref{lemma:bialgebraM}, we have a well-defined bialgebra map $f: A\bullet A^!\to D$ given by $f(z^k_j)=d_{kj}$. Moreover, it is clear that $f$ is the unique bialgebra map such that $(\id_D\otimes f)\circ \rho=\rho_D$. Therefore, by the universal property, we have $\underline{\rm end}^r(A)\cong A\bullet A^!$. Similarly, one can prove that $\underline{\rm end}^l(A^!)\cong A^!\bullet A$.

(2): Similar to the construction of Manin's universal quantum groups associated to quadratic algebras, we show that the Hopf envelope of $\underline{\rm end}^r(A)$ exists with universal property. 

(3): It follows from (1) and (2) since $(A^!)^! \bullet A^! \cong A \bullet A^!$ as bialgebras. 

Finally, (4) and (5) can be proved by a similar argument to \cite[Lemma 6.6]{Manin2018}.
\end{proof}

\subsection{Two-cocycle twists of Manin's universal quantum groups}
\label{subsec:2cocycle}

In this subsection, we investigate the behavior of a 2-cocycle twist of a Manin's universal quantum group. We first recall the 2-cocycle twist introduced by Doi and Takeuchi \cite{Doi93,DT94}.

\begin{defn}\cite[Definition 7.7.1]{Radford2012}
\label{defn:cocycle} 
Let $B$ be a bialgebra over $\kk$.
A \emph{left 2-cocycle} on $B$ is a convolution invertible bilinear map $\sigma: B \times B \rightarrow \kk$ so that 
\[
    \sum \sigma(x_1, y_1)\, \sigma (x_2 y_2, z) = \sum \sigma(y_1,z_1)\, \sigma(x,y_2z_2), 
\]
for all $x,y,z \in B$. Similarly, \emph{a right 2-cocycle} on $B$ is a convolution invertible bilinear map $\sigma: B \times B \rightarrow \kk$ satisfying 
\[
    \sum \sigma(x_2, y_2)\, \sigma (x_1 y_1, z) = \sum \sigma(y_2,z_2)\, \sigma(x,y_1z_1).
\]
For any 2-cocycle $\sigma$, we denote its convolution inverse by $\sigma^{-1}$.
\end{defn}

Throughout, for simplicity, we assume all 2-cocycles are normal, that is, \[\sigma(x,1) = \sigma(1,x) = \varepsilon(x),\] 
for all $x \in B$. Note that if $\sigma$ is a right 2-cocycle on a Hopf algebra $H$, then $\sigma^{-1}$ is a left 2-cocycle on $H$ \cite[Lemma 7.7.2]{Radford2012}.
Given a right 2-cocycle $\sigma: B \times B \to \kk$ on a bialgebra $B$, let $B^\sigma = B$ as a  coalgebra, endowed with the original unit and deformed product
\begin{align}\label{def:2Hopf}
 x*_\sigma y:=\sum \sigma^{-1}(x_1,y_1)\,x_2y_2\,\sigma(x_3,y_3),
\end{align}
for any $x,y\in B$. In fact, for a Hopf algebra $H$ with antipode $S$, $H^\sigma$ is again a Hopf algebra with deformed antipode $S^{\sigma}$ \cite[Theorem 1.6]{Doi93} and \cite[Proposition 7.7.3]{Radford2012}. We call $H^\sigma$ the \emph{2-cocycle twist of $H$} by $\sigma$. It is well-known that two Hopf algebras are 2-cocycle twists of each other if and only if there exists a bicleft object between them (see e.g., \cite{Sch1996}). 

Suppose $A$ is a right $H$-comodule algebra via coaction $a \mapsto \sum a_0 \otimes a_1$ for any $a \in A$. For any right 2-cocycle $\sigma$ on $H$, one can define a \emph{twisted algebra} $A_{\sigma}$ of $A$ by $\sigma$ as follows. As vector spaces $A_{\sigma}=A$ and the product of $A_{\sigma}$ is defined to be
\begin{align}\label{def:twistp}
a\cdot_\sigma b=\sum \sigma(a_1,b_1)\,a_0b_0, \qquad \text{for all } a,b\in A.
\end{align}
We call $A_{\sigma}$ the \emph{2-cocycle twist of $A$} by $\sigma$. Note that $A_\sigma$ is naturally a right $H^\sigma$-comodule algebra.

\begin{defn}\cite{Sch1996} \label{defn:biGalois}
Let $H$ and $K$ be two Hopf algebras. A \emph{left $H$-Galois object} is a left $H$-comodule algebra $T$ such that the linear map
\[
T\otimes T\xrightarrow{\rho\otimes \id} H\otimes T\otimes T\xrightarrow{\id \otimes m} H\otimes T
\]
is bijective, where $\rho: T\to H\otimes T$ and  $m: T\otimes T\to T$ are the corresponding left $H$-comodule algebra structure maps on $T$. A \emph{right $K$-Galois object} can be defined analogously and an \emph{$H$-$K$-bi-Galois object} is an $H$-$K$-bicomodule algebra that is both a left $H$-Galois object and a right $K$-Galois object. 
\end{defn}

In particular, when we consider $H$ as a right comodule algebra over itself via its comultiplication, the twisted algebra $H_\sigma$ is an $H$-$H^\sigma$-bi-Galois object. As a consequence, $H$ and $H^\sigma$ are Morita--Takeuchi equivalent under the monoidal equivalence 
\begin{align} \label{eq:equiv2}
(F,\,\xi): {\rm comod}(H)~\stackrel{\otimes}{\cong}~{\rm comod}(H^\sigma),
\end{align}
where $F$ is defined by $F(U):=U_{\sigma}=U$ as vector space and 
$\xi$ is the natural isomorphism
\begin{align} \label{eq:monoidalstructure}
\xi_{U,V}: (U\otimes V)_{\sigma}~&\to~ U_{\sigma}\otimes_{\sigma} V_{\sigma}\\
u\otimes v~&\mapsto~ \sum \sigma^{-1}(u_1,v_1)\,u_0\otimes v_0, \notag
\end{align}
with inverse $\xi_{U,V}^{-1}: u\otimes v\mapsto \sum \sigma(u_1,v_1)u_0\otimes v_0$, for any $u \in U, v \in V$.

We shall show that a 2-cocycle twist of Manin's universal quantum group of a locally finite $\mathbb Z$-graded algebra $A$ by $\sigma$ is isomorphic to Manin's universal quantum group of the 2-cocycle twist $A_\sigma$. Consequently, these two quantum groups are Morita--Takeuchi equivalent.

\begin{thm}
\label{thm:action-Manin}
Let $A$ be any $\mathbb Z$-graded locally finite algebra and $H$ be any Hopf algebra that right coacts on $A$ preserving its grading. Then for any right 2-cocycle $\mu$ on $ H$, there is a right 2-cocycle $\sigma$ on $\underline{\rm aut}^r(A)$ such that there is an isomorphism between 2-cocycle twists
\[A_{\sigma}~\cong~A_{\mu},\]
as graded algebras. Additionally, we have 
\[\underline{\rm aut}^r(A_\mu)~\cong~ \underline{\rm  aut}^r(A)^\sigma\]
as Hopf algebras.
\end{thm}

\begin{proof}
By the universal property of $\underline{\rm aut}^r(A)$, the graded right coaction of $H$ on $A$ uniquely factors through the universal graded right coaction of $\underline{\rm aut}^r(A)$ via a Hopf algebra map $f:\underline{\rm aut}^r(A)\to H$. We write the coaction $A\to A\otimes H$ as $x\mapsto \sum x_0\otimes x_1$ and the universal coaction $A\to A\otimes \underline{\rm aut}^r(A)$ as  $x\mapsto \sum x_0'\otimes x_1'$ for any $x\in A$ and we obtain $\sum x_0'\otimes f(x_1')=\sum x_0\otimes x_1$. Define
\[
\sigma:=\left(\underline{\rm aut}^r(A)\otimes \underline{\rm aut}^r(A)\xrightarrow{f\times f} H\times H\xrightarrow{\mu} \kk\right).
\]
It is direct to check that $\sigma$  is a right 2-cocycle on $\underline{\rm aut}^r(A)$ since $f$ is a Hopf algebra map. By definition, $A=A_\sigma=A_{\mu}$ as graded vector spaces and 
\[x\cdot_{\sigma}y=\sum \sigma(x_1', y_1')x_0'y_0'=\sum \mu(f(x_1'),f(y_1'))x_0'y_0'=\sum \mu(x_1, y_1)x_0y_0=x\cdot_{\mu}y\]
for any $x,y\in A$. So the identity map on the vector space $A$ gives an isomorphism  from $A_\sigma$ to $A_\mu$ as $H^\sigma$-comodule algebras. 

Finally, using an argument similar to \cite[Theorem 4.2.9]{HNUVVW2}, one can show that $\underline{\rm aut}^r(A_\sigma)$ and $\underline{\rm  aut}^r(A)^\sigma$ are isomorphic by using the universal property of Manin's universal quantum groups. 
\end{proof}

\subsection{Zhang twists vs. 2-cocycle twists of graded algebras}
\label{subsec:compare}

In this subsection, we use the notion of twisting pairs formulated from graded automorphisms of $A$, as done in \cite{HNUVVW21}, to construct 2-cocycle twists of $\underline{\rm aut}^r(A)$. Recall that if $A$ is a graded algebra and $\phi: A \to A$ is a graded automorphism of $A$, the \emph{right Zhang twist of $A$}, denoted by $A^{\phi}$, has the same graded vector space as $A$ with deformed product 
\[a \cdot_{\phi} b=a\phi^{|a|}(b),\] 
for any homogeneous elements $a,b\in A$, where $|a|$ denotes the degree of $a$. 

To understand the relationship between a Zhang twist $A^{\phi}$ and a 2-cocycle twist $A_{\sigma}$,  our strategy is to consider $A^{\phi}$ and $A_{\sigma}$ as comodule algebras under the coaction of the same Hopf algebra $H$. By \cite[Theorem E]{HNUVVW21}, we shall form a twisting pair from $\phi$, which can be used to determine a 2-cocycle. To that aim, we state the following definition and conditions in order to obtain this 2-cocycle twist of $H$.

\begin{defn}\cite[Definition B and D]{HNUVVW21}
\label{def:conditions}
Let $(B, m, u, \Delta, \varepsilon)$ be a bialgebra. A pair ($\phi_1, \phi_2$) of algebra automorphisms of $B$ is said to be a \emph{twisting pair} if the following conditions
hold:
\begin{enumerate}
    \item[(\textbf{P1})] $\Delta \circ  \phi_1 = (\id \otimes \phi_1)\circ \Delta$ and $\Delta \circ\phi_2 = (\phi_2 \otimes \id) \circ\Delta$; 
    \item [(\textbf{P1})] $\varepsilon\circ \phi_1\circ \phi_2=\varepsilon$. 
\end{enumerate}
\label{defn:twisting conditions}
Moreover, $B$ satisfies the \emph{twisting conditions} if 
\begin{enumerate}
    \item[(\textbf{T1})] as an algebra $B=\bigoplus_{n\in \mathbb{Z}} B_n$ is $\mathbb{Z}$-graded, and
    \item[(\textbf{T2})] the comultiplication satisfies $\Delta(B_n)\subseteq B_n\otimes B_n$, for all $n\in \mathbb Z$. 
\end{enumerate}
\end{defn}

As pointed out in \cite[Corollary 1.1.10]{HNUVVW21}, if $H=\bigoplus_{n\in \mathbb Z} H_n$ is a Hopf algebra that satisfies the twisting conditions, then $S(H_n)\subseteq H_{-n}$ for all $n\in \mathbb Z$. An equivalent description of twisting conditions on $B$ was given by Bichon et.~ al.~ in \cite[Lemma 1.3]{Bichon-Neshveyev-Yamashita2016} including the existence of a cocentral homomorphism from $B$ to the group algebra $\kk[\mathbb Z]$. Recall that a homomorphism of coalgebras $f: C\to D$ is said to be \emph{cocentral} if $\sum f(c_1)\otimes c_2=\sum f(c_2)\otimes c_1 \in D\otimes C$, for all $c \in C$. 

Here, we show that Manin's universal quantum group, indeed, satisfies the twisting conditions.

\begin{lemma}
\label{TCautA}
Let $A=\bigoplus_{n\in \mathbb Z} A_n$ be a $\mathbb Z$-graded algebra that is locally finite, and $\ara$ be its associated right universal quantum group via the universal coaction $\rho: A\to A\otimes \underline{\rm aut}^r(A)$. Then $\underline{\rm aut}^r(A)=\bigoplus_{n\in \mathbb Z} \left(\underline{\rm aut}^r(A)\right)_n$ has a natural $\mathbb Z$-grading such that $\rho: A_n\to A_n\otimes\left(\underline{\rm aut}^r(A)\right)_n$, for each $n \in \mathbb Z$. Moreover, $\underline{\rm aut}^r(A)$ satisfies the twisting conditions. Similar results hold for $\underline{\rm aut}^l(A)$.
\end{lemma}

\begin{proof}
Recall that $\uend^r(A)$ is the universal bialgebra with right coaction $\rho$ on $A$ preserving the $\mathbb Z$-grading of $A$ with universal property satisfying \eqref{def:aut}. By \Cref{lem:ManinM}, $\uaut^r(A)$ is the Hopf envelope of $\uend^r(A)$. Hence, by \cite[Lemma 2.1.10]{HNUVVW21}, it suffices to prove the lemma for $\uend^r(A)$.

Now by \cite[Lemma 1.3]{Bichon-Neshveyev-Yamashita2016}, it is enough to show that there is a cocentral bialgebra map $f: \underline{\rm end}^r(A)\to \kk[\mathbb Z]$. Note that since $A$ is $\mathbb Z$-graded, it follows that $A$ is a graded right $\kk[\mathbb Z]$-comodule algebra via $a \mapsto a\otimes g^n$ for any $a\in A_n$ and $n \in \mathbb Z$, where $\mathbb Z=\langle g\rangle$. By the universal coaction $\rho: A\to A\otimes \underline{\rm end}^r(A)$, where $\rho(a)=\sum a_{0}\otimes a_1$, we get a unique bialgebra map $f: \underline{\rm end}^r(A)\to \kk[\mathbb Z]$ such that 
\begin{align}\label{lem:eq211}
\sum a_0\otimes f(a_1)=a\otimes g^n, \quad \text{ for all } a\in A_n.
\end{align}
It remains to show that $f: \underline{\rm end}^r(A)\to \kk[\mathbb Z]$ is cocentral. Denote by $V_n$ the cosupport of the left comodule $A_n$ in $\underline{\rm end}^r(A)$, that is, the smallest subspace $V_n$ of $\underline{\rm end}^r(A)$ such that $\rho(A_n)\subseteq A_n\otimes V_n$. It is straightforward to check that $V_n$ is a subcoalgebra of  $\underline{\rm end}^r(A)$. Since $\era$ coacts on $A$ inner-faithfully, $\underline{\rm end}^r(A)$ is generated by $\{V_n\}_{n\in \mathbb Z}$ as an algebra. Therefore, it is enough to show that 
\begin{align*}
    \sum f(v_1)\otimes v_2=\sum f(v_2)\otimes v_1, \quad \text{ for all } v\in V_n.
\end{align*}
Using the fact that the $V_n$-comodule structure on $A_n$ satisfies \eqref{lem:eq211}, we have
$\sum f(v_1)\otimes v_2=g^n\otimes v$ and so $f(v)=g^n\varepsilon(v)$ by (2.9).

Hence 
\begin{align*}
    \sum f(v_1)\otimes v_2=g^n\otimes v=\sum f(v_2)\otimes v_1.
\end{align*}
It follows that $f$ is cocentral. A similar argument holds for $\underline{\rm end}^l(A)$ and $\underline{\rm aut}^l(A)$.
\end{proof}

The following result shows that graded automorphisms of $A$ are naturally in one-to-one correspondence with 2-cocycles on Manin's universal quantum group $\underline{\rm aut}^r(A)$ given by twisting pairs.

\begin{thm}
\label{lem:2cocycleGrade}
Let $A$ be a connected graded algebra, finitely generated in degree one subject to $m$-homogeneous relations. The following groups are isomorphic.
\begin{itemize}
    \item[(1)] the group of graded automorphisms of $A$;
    \item[(2)] the group of twisting pairs of $\underline{\rm aut}^r(A)$ under component-wise composition (see \cite[Corollary 2.1.3]{HNUVVW21});
    \item[(3)] the group of one-dimensional representations of $\underline{\rm aut}^r(A)$ under tensor product of $\underline{\rm aut}^r(A)$-modules. 
\end{itemize}
Moreover, for any graded automorphism $\phi$ of $A$, let $\sigma$ be the right 2-cocycle on $\uaut^r(A)$ given by the twisting pair associated to $\phi$ (see (\ref{eq:2cocycleA}) below). Then \[A^{\phi}\cong A_{\sigma}\] as graded algebras. 
\end{thm}

\begin{proof}
(1) $\Leftrightarrow$ (2): Let $\phi$ be a graded automorphism of $A$. We denote by $\phi^!$ the induced graded automorphism on the $m$-Koszul dual algebra $A^!$. By \Cref{lem:ManinM}(4)-(5), we have a unique pair of graded algebra automorphisms
\begin{align}\label{eq:twistpairB}
\left(\underline{\rm end}^r(\phi),\ \underline{\rm end}^l((\phi^{-1})^!))\right)
\end{align}
of $\underline{\rm end}^r(A)$ making the following diagrams 
\begin{align}\label{lem:cend}
\xymatrix{
A\ar[r]^-{\rho}\ar[d]_-{\phi} & A\otimes \underline{\rm end}^r(A)\ar[d]^-{\id\otimes\underline{\rm end}^r(\phi) }\\
A\ar[r]^-{\rho} & A\otimes \underline{\rm end}^r(A)
}
\quad\quad 
\xymatrix{
A^!\ar[r]^-{\rho^!}\ar[d]_-{(\phi^{-1})^!} &  \underline{\rm end}^r(A)\otimes A^!\ar[d]^-{\underline{\rm end}^l((\phi^{-1})^!)) \otimes \id}\\
A^!\ar[r]^-{\rho^!} &  \underline{\rm end}^r(A)\otimes A^!
}
\end{align} 
commute. By an identical proof of \cite[Lemma 3.2.3]{HNUVVW21}, we know that \eqref{eq:twistpairB} is a twisting pair of $\underline{\rm end}^r(A)$ in terms of \Cref{def:conditions}. 
Since $\underline{\rm aut}^r(A)$ is the Hopf envelope of $\underline{\rm end}^r(A)$, we apply \cite[Proposition 2.1.13]{HNUVVW21} to conclude that the twisting pair \eqref{eq:twistpairB} extends uniquely to that of $\underline{\rm aut}^r(A)$ via the natural bialgebra map $\underline{\rm end}^r(A)\to \underline{\rm aut}^r(A)$. We denote this uniquely extended twisting pair of $\underline{\rm aut}^r(A)$ by
\begin{align}
\label{eq:twistpairA}
\left(\underline{\rm aut}^r(\phi),\ \underline{\rm aut}^l((\phi^{-1})^!))\right)
\end{align}
Moreover, the diagrams 
 \begin{align}\label{lem:caut}
\xymatrix{
A\ar[r]^-{\rho}\ar[d]_-{\phi} & A\otimes \underline{\rm aut}^r(A)\ar[d]^-{\id\otimes\underline{\rm aut}^r(\phi) }\\
A\ar[r]^-{\rho} & A\otimes \underline{\rm aut}^r(A)
}
\quad\quad 
\xymatrix{
A^!\ar[r]^-{\rho^!}\ar[d]_-{(\phi^{-1})^!} &  \underline{\rm aut}^r(A)\otimes A^!\ar[d]^-{\underline{\rm aut}^l((\phi^{-1})^!)) \otimes \id}\\
A^!\ar[r]^-{\rho^!} &  \underline{\rm aut}^r(A)\otimes A^!
}
\end{align}
commute since the universal coactions of $\underline{\rm aut}^r(A)$ on $A$ and $A^!$ factor through $\underline{\rm end}^r(A)$. 

Conversely, suppose $(\phi_1,\phi_2)$ is a twisting pair of $\underline{\rm aut}^r(A)$. By \cite[Lemma 3.2.3]{HNUVVW21}, we define a graded automorphism $\phi$ of $A$ such that
\[
\phi(a)=\sum a_0\varepsilon(\phi_1(a_1))\quad \text{with inverse}\quad \phi^{-1}(a)=\sum a_0\varepsilon(\phi_2(a_1)),
\] for any homogeneous element $a\in A$. Finally, according to \eqref{lem:caut}, it is straightforward to check that the above assignments give a one-to-one correspondence, which indeed respects the group structures. 

\medskip

(2) $\Leftrightarrow$ (3): This follows from \cite[Theorem  2.1.7]{HNUVVW21}. That is, for any twisting pair $(\phi_1,\phi_2)$ of $\underline{\rm aut}^r(A)$, we define a one-dimensional representation of $\underline{\rm aut}^r(A)$ via the algebra map $\varepsilon\circ \phi_1: \underline{\rm aut}^r(A)\to \kk$. Conversely, for any one-dimensional representation  of $\underline{\rm aut}^r(A)$ via an algebra map $\pi: \underline{\rm aut}^r(A)\to \kk$, the corresponding twisting pair $(\phi_1,\phi_2)$ may be viewed as   right and left winding automorphisms given by $\pi$ (see e.g., \cite[\S 2.5]{BZ08}).

Now for the last assertion, we consider the twisting pair \eqref{eq:twistpairA} associated to $\phi$ satisfying commutative diagrams \eqref{lem:caut}. By \cite[Proposition 2.3.1]{HNUVVW21}, we have a right 2-cocycle $\sigma$ on $\underline{\rm aut}^r(A)$ with convolution inverse $\sigma^{-1}$ via
\begin{align}
\label{eq:2cocycleA}
\sigma(x,y)=\varepsilon(x)\, \varepsilon((\underline{\rm aut}^r(\phi))^{|x|}(y)) \qquad {\rm and} \qquad \sigma^{-1}(x,y)=\varepsilon(x)\, \varepsilon((\underline{\rm aut}^l(\phi^{-1})^{!})^{|x|}(y)),
\end{align}
for any homogeneous elements $x,y\in \underline{\rm aut}^r(A)$. 
Since $A$ is a comodule algebra over $\underline{\rm aut}^r(A)$, the twist $A_{\sigma}$ is a comodule algebra over $\underline{\rm aut}^r(A)^\sigma$. One can see that $A^\phi=A_{\sigma}=A$ as graded vector spaces, and so it suffices to show that $A^\phi\cong A_{\sigma}$ as graded algebras. For any two homogeneous elements $a,b \in A$, we have
\begin{align*}
    a \cdot_{\sigma} b&=\, \sum a_0b_0\, \sigma(a_1, b_1) \\
    \overset{\eqref{eq:2cocycleA}}&{=}\, \sum a_0b_0\,\varepsilon(a_1)\varepsilon(\underline{\rm aut}^r(\phi)^{|a_1|}(b_1))\\
    &=\,\sum a\, b_0\,\varepsilon(\underline{\rm aut}^r(\phi)^{|a|}(b_1))\\
    &=\,a\,(\id \otimes \varepsilon)\circ\left ((\id \otimes \underline{\rm aut}^r(\phi)^{|a|})\circ \rho\right)(b)\\
     \overset{\eqref{lem:caut}}&{=}a\,(\id \otimes \varepsilon)\circ\left(\rho \circ \phi^{|a|}\right)(b)\\
     &=\,a\,\phi^{|a|}(b)\\
    & =\,a \cdot_\phi b,
\end{align*} 
where the second to last equality follows from the counit axiom of the comodule structure $\rho$. Therefore, the $\kk$-linear identity map $\id$ preserves the multiplication and so $A^{\phi}\cong A_{\sigma}$ as desired.
\end{proof}

This is a comodule algebra analogue of \cite[Theorem 3.2.5 and Theorem 3.2.6]{HNUVVW21} that a Zhang twist of $A$ can be realized as a 2-cocycle twist of $A$.

\begin{Cor}
Let $A$ and $B$ be two connected graded algebras finitely generated in degree one, both subject to $m$-homogeneous relations. If the categories of modules over $\underline{\rm aut}^r(A)$ and over $\underline{\rm aut}^r(B)$ are monoidally equivalent, then the groups of graded automorphisms of $A$ and of $B$ are isomorphic. 
\end{Cor}
\begin{proof}
    It is a consequence of the bijection between (1) and (3) stated in \Cref{lem:2cocycleGrade}.
\end{proof}

\section{Homological regularities under Morita--Takeuchi equivalence}
\label{sec:properties}

In this section, we study various homological properties of connected graded algebras that remain invariant under Morita--Takeuchi equivalence. We say that two Hopf algebras $H$ and $K$ are \emph{Morita--Takeuchi equivalent} if there is some monoidal equivalence described by $(F,\xi)$, where   \[F: {\rm comod}(H)\stackrel{\otimes}{\cong}{\rm comod}(K)
\]
is an equivalence of categories, together with a natural isomorphism 
\begin{align*}\label{eq:monoidaldata}
 \xi^{-1}_{V,W}: F(V)\otimes F(W)\cong F(V\otimes W),
 \end{align*}
for any two $H$-comodules $V$ and $W$ satisfying some compatible conditions in \cite[Definition 2.4.1]{EGNO}. Following the proof of \cite[Proposition 1.16]{Bichon2014}, $(F,\xi^{-1})$ restricts to a monoidal equivalence between ${\rm comod}_{\rm fd}(H)$ and  ${\rm comod}_{\rm fd}(K)$. In particular, $F$ sends one-dimensional $H$-comodules to one-dimensional $K$-comodules. 

Moreover, according to \cite{Sch1996} Morita--Takeuchi equivalence between $H$ and $K$ can be described in terms of bi-Galois objects (see \Cref{defn:biGalois}). As a result, any monoidal equivalence $(F,\xi^{-1})$ can be explicitly given by the cotensor product 
\begin{equation}\label{eq:MTE} 
\begin{aligned}
F: {\rm comod}(H)~&\stackrel{\otimes}{\cong}~{\rm comod}(K) \\
M~&\mapsto~ \widetilde{M}:= M\square_HT 
\end{aligned}
\end{equation}
for some $H$-$K$-bi-Galois object $T$, together with the natural isomorphism (as in \cite[Proposition 1.4]{Ulbrich1987} for left comodules)
\begin{align*}
\xi^{-1}_{V,W}: (V\square_HT)\otimes (W\square_HT)~&\cong~ (V\otimes W)\square_HT\\
\left(\sum v_i\otimes a_i\right)\otimes \left(\sum w_j\otimes b_j\right)~&\mapsto~ \sum v_i\otimes w_j\otimes a_ib_j.
\end{align*}
 The cotensor product is given by $M\square_HT:=\{x\in M\otimes T\,|\, (\rho_M\otimes \id_T)(x)=(\id_M \otimes \rho_T)(x)\}$, where $\rho_M: M\to M\otimes H$ and $\rho_T: T\to H\otimes T$ are the corresponding $H$-comodule structure maps on $M$ and $T$, respectively. Here, the right coaction of $K$ on $\widetilde{M}:= M\square_HT$ comes from that of $K$ on $T$, since $T$ is an $H$-$K$-bi-Galois object. Note that we can also view $T$ as a right $H$-comodule using the antipode of $H$. In this case, $\widetilde{M}=M\square_HT=(M\otimes T)^{\co H}$, which we will use later.

Throughout this section, let $A=\bigoplus_{i\ge 0}A_i$ be a connected graded right $H$-comodule algebra. It is clear that $\widetilde{A}=\bigoplus_{i\ge 0}\widetilde{A_i}$ is a connected graded right $K$-comodule algebra since $\widetilde{(-)}$ is a monoidal functor. We will show that $A$ and $\widetilde{A}$ share some similar homological properties. In particular, we show that the numerical notions of Tor-regularity and Castelnuovo--Mumford regularity (CM-regularity), as well as AS-regularity are Morita--Takeuchi invariant in \Cref{subsec:numerical regularity}.

\subsection{The category of graded relative $(A,H)$-modules}
\label{subsec:relative modules}

In this subsection, we review the definition of a graded relative $(A,H)$-module and prove some basic properties, which are essential in various homological computations related to $A$.  

A vector space $M$ is called a \emph{left graded relative $(A,H)$-module} if $M=\bigoplus_{i\in \mathbb Z} M_i$ is a graded left $A$-module and each homogeneous piece $M_i$ is a right $H$-comodule such that 
\begin{equation}
\label{eq:relative}
\sum (ax)_0\otimes (ax)_1=\sum a_0x_0\otimes a_1x_1,
\end{equation}
for all homogeneous $a\in A$ and $x\in M$. A \emph{right graded relative $(A,H)$-module} refers to a simultaneously graded right $A$-module and right $H$-comodule with the $H$-coaction preserving the grading, which satisfies $\sum (xa)_0\otimes (xa)_1=\sum x_0a_0\otimes x_1a_1$. In both cases, we use the right $H$-comodule structure, with respective left or right $A$-module structure. In what follows, by graded relative $(A,H)$-modules, we mean left graded relative modules, unless stated otherwise.

We call a left graded relative $(A,H)$-module $M$ \emph{(finite) free} if $M \cong A \otimes V$ for some (finite-dimensional) graded right $H$-comodule $V$, with $A$-action given by multiplication on the left tensorand, and with $H$-coaction coming from the standard definition by tensor product of $H$-comodules. Note that any free left graded relative $(A,H)$-module is a free graded left $A$-module when we disregard its $H$-comodule structure. Free right graded relative $(A,H)$-modules can be similarly defined. 

The category of all left graded relative $(A,H)$-modules is denoted by $\,_{\Gr(A)}\mathcal M^H$ consisting of all left graded relative $(A,H)$-modules together with degree-preserving $A$-linear and $H$-colinear maps. Similarly, we denote the category of all right graded relative $(A,H)$-modules by $\mathcal M^H_{\Gr(A)}$.

\begin{lemma}
\label{free-res-rel}
Let $A$ be a connected graded right $H$-comodule algebra. Then every graded relative $(A,H)$-module $M$ has a resolution consisting of free graded relative $(A,H)$-modules. In particular, if $A$ is noetherian and $M$ is finitely generated as an $A$-module, then the resolution of $M$ can be chosen to consist of finite free graded relative $(A,H)$-modules.
\end{lemma} 

\begin{proof}
It suffices to show that for every graded relative $(A,H)$-module $M$, there are some graded $H$-comodule $V$ and some surjective graded relative $(A,H)$-module map from $A\otimes V$ to $M$. Then the desired resolution can be constructed from $A \otimes V 	\twoheadrightarrow M$. Specifically, we can take $V=M$ and the surjection $A\otimes M \twoheadrightarrow M$ to be the $A$-module structure map on $M$, which is both graded $A$-linear and $H$-colinear.  

Suppose $A$ is noetherian and $M$ is finitely generated over $A$. Then $M$ has a finite set $W$ of homogeneous generators over $A$. By the fundamental theorem of comodules, $W$ is contained in some finite-dimensional graded subcomodule $V$ of $M$. Then we can take the surjection $A\otimes V\to M$ induced again by the $A$-module structure map on $M$. 
\end{proof}

\begin{remark}
\label{rem:trivial}
When $A$ is connected graded, we can view $A$ as a left graded relative $(A,H)$-module with left $A$-module action given by the multiplication in $A$. It is clear that $A$ has graded relative $(A,H)$-submodules $A_{\ge n}$ and quotient modules $A/A_{\ge n}$, for every $n\ge 1$. In particular, if $A$ is noetherian, all the above relative $(A,H)$-modules are finitely generated over $A$. 
\end{remark}

Let $M=\bigoplus_{i\in \mathbb Z}M_i$ and $N=\bigoplus_{i\in \mathbb Z}N_i$ be two graded relative $(A,H)$-modules. Considering $M$ and $N$ as graded modules over $A$, we use $\Hom_A^d(M,N)$ to denote the set of all $A$-linear homogeneous homomorphisms $f: M\to N$ of degree $d$ such that $f(M_i)\subseteq N_{i+d}$ for any $i$. We set $\underline{\Hom}_A(M,N):=\bigoplus_{d\in \mathbb Z}\Hom_A^d(M,N)$, and we denote the corresponding derived functors by $\underline{\Ext}_A^\bullet(M,N)$. If $M$ has a resolution in $\,_{\Gr(A)}{\mathcal M}^H$ consisting of finite free graded relative $(A,H)$-modules, then $\underline{\Ext}_A^\bullet (M,N)=\Ext_A^\bullet(M,N)$, but this is not true in general. 
Moreover, we use $\underline{\Hom}^H(M,N)$ and   $\underline{\Hom}_A^H(M,N)$ to denote the spaces of all homogeneous $H$-colinear and both $H$-colinear and $A$-linear maps from $M$ to $N$ of arbitrary degree, respectively. 

Suppose $H$ has a bijective antipode $S$. We set the degree of $N_i\otimes H$ to be that of $N_i$, for $i\in \mathbb Z$. Consider the graded map $\rho: \underline{\Hom}_\kk(M,N)\to  \underline{\Hom}_\kk(M,N\otimes H)$ given by
\begin{align}\label{eq:coactionHom}
\rho(f)(m)=\sum f(m_0)_0\otimes S^{-1}(m_1)f(m_0)_1
\end{align}
for any homogeneous map $f\in \underline{\Hom}_\kk(M,N)$ and any homogeneous $m\in M$. Since $\underline{\Hom}_\kk(M,N)\otimes H\subset \underline{\Hom}_\kk(M,N\otimes H)$, we introduce 
\[
\underline{\HOM}_\kk(M,N):=\{f\in \underline{\Hom}_\kk(M,N)\,|\, \rho(f)\in \underline{\Hom}_\kk(M,N)\otimes H\}. 
\]
It is well-known (see e.g., \cite{Stefan99,Ul90}) that $\rho$ gives $\underline{\HOM}_\kk(M,N)$ a graded right $H$-comodule structure which is the largest graded $H$-comodule contained in $\underline{\Hom}_\kk(M,N)$.

We summarize some results from \cite{CG05} below, which we will employ to compute certain homology groups of relative $(A,H)$-modules. 

\begin{lemma}
\label{lem:relHom}
Let $A$ be a connected graded right $H$-comodule algebra and let $M$ and $N$ be two left graded relative $(A,H)$-modules. Suppose $H$ has a bijective antipode and $M$ is finitely generated over $A$. Then we have the following:
\begin{itemize}
    \item[(1)] $\underline{\Hom}_A(M,N)$ is a graded right $H$-subcomodule of $\underline{\HOM}_\kk(M,N)$.
    \vspace{0.05in} 
    \item[(2)] $\underline{\Hom}_A^H(M,N)\cong \underline{\Hom}_A(M,N)^{\co H}$ as vector spaces.
    \item[(3)] Suppose $N$ is also finitely generated over $A$. Then the natural composition map 
    \[\underline{\Hom}_A(N,P)\otimes \underline{\Hom}_A(M,N)\to \underline{\Hom}_A(M,P)\] is graded $H^{\op}$-colinear for any graded relative $(A,H)$-module $P$. Here, $H^{\op}=H$ as a vector space, with opposite multiplication.
    \item[(4)] Suppose $M$ can be presented by finite free graded relative $(A,H)$-modules. Then for any graded right $H$-comodule $R$
    \[\underline{\Hom}_A(M,N\otimes R)\cong \underline{\Hom}_A(M,N)\otimes R\]
    as graded right $H$-comodules, where the $A$-module structure on $N\otimes R$ is induced by that on $N$.
    \end{itemize}
\end{lemma}

\begin{proof}
(1) and (2): The results are analogues of \cite[Lemma 2.2]{CG05} in the graded case.

(3): By (1), all the $\underline{\Hom}$-spaces in the statement are graded right $H$-comodules. Let $f\in \underline{\Hom}_A(N,P)$ and $g\in \underline{\Hom}_A(M,N)$ be two homogeneous $A$-linear maps. We follow the proof of \cite[Corollary 1.6]{CG05} to write the composition $f\circ g$ as below. By \cite[Proposition 1.5]{CG05}, there exist two finite-dimensional graded $H$-comodules $V$ and $W$, two homogeneous elements $v\in V$, $w\in W$, and two graded $H$-colinear maps $F: N\otimes W\to P$, $G: M\otimes V\to N$ such that $f(n)=F(n\otimes w)$ and $g(m)=G(m\otimes v)$ for all homogeneous $m\in M$ and $n\in N$. As a consequence, $(f\circ g)(m)=F(G(m\otimes v)\otimes w)$. Moreover, the proof of \cite[Proposition 1.5]{CG05} shows that the $H$-comodule structures on $f$, $g$ and $f\circ g$ are explicitly given by
\begin{align*}
\rho(f)&=\sum F(-\otimes w_0)\otimes w_1\in \underline{\Hom}_A(N,P)\otimes H^{\op}\\
\rho(g)&=\sum G(-\otimes v_0)\otimes v_1\in \underline{\Hom}_A(M,N)\otimes H^{\op}\\
\rho(f\circ g)&=\sum F(G(-\otimes v_0)\otimes w_0)\otimes v_1w_1\in \underline{\Hom}_A(M,P)\otimes H^{\op}.
\end{align*}
It is then straightforward to show that $\rho(f\circ g)=\rho(f)\rho(g)$ using the fact that the multiplication in $H^{\op}$ is the opposite of the multiplication of $H$. Hence, the result follows. 

(4): We claim that the graded map $\Psi_M: \underline{\Hom}_A(M,N)\otimes R \to \underline{\Hom}_A(M,N\otimes R)$ given by $\Psi_M(f\otimes r)(m)=f(m)\otimes r$, for any homogeneous $f\in \underline{\Hom}_A(M,N)$, $r\in R$ and $m\in M$, is $H$-colinear and natural in $M$. By \cite[Proposition 1.5]{CG05}, there exist a finite-dimensional graded $H$-comodule $V$, a homogeneous element $v\in V$, and a graded $H$-colinear map $F: M\otimes V\to N$ such that $f(m)=F(m\otimes v)$. It follows that  $\Psi_M(f\otimes r)=\Psi_M(F(-\otimes v)\otimes r)=F(-\otimes v)\otimes r$. By the fundamental theorem of comodules, $r$ is contained in some finite-dimensional graded $H$-subcomodule $W$ of $R$. Denote by $\widetilde{F}=F\otimes {\iota}: M\otimes V\otimes W\to N\otimes R$ the graded $H$-colinear map, where $\iota: W\to R$ is the natural embedding. Hence $\Psi_M(f\otimes r)=\widetilde{F}(-\otimes v\otimes r)$. Note that the $H$-comodule structure on $f$ is given by $\rho(f)=\sum F(-\otimes v_0)\otimes v_1$ and the $H$-comodule structure on $\Psi_M(f\otimes r)$ is given by $\rho(\Psi_M(f\otimes r))=\sum \widetilde{F}(-\otimes v_0\otimes r_0)\otimes v_1r_1$. It is straightforward to check that $\Psi_M$ is $H$-colinear. Moreover, $\Psi_M$ is natural in $M$ by its explicit construction.  

Now let $P_1\to P_0\to M\to 0$
be a presentation for $M$ consisting of finite free graded relative $(A,H)$-modules. We apply the functors $\underline{\Hom}_A(-,N)\otimes R$ and $\underline{\Hom}_A(-,N\otimes R)$ to the above presentation and use $\Psi$. By the discussion above, we have the following commutative diagram 
\[
\xymatrix{
0\ar[r] & \underline{\Hom}_A(M,N)\otimes R \ar[r]\ar[d]^-{\Psi_{M}}& \underline{\Hom}_A(P_0,N)\otimes R \ar[r]\ar[d]^-{\Psi_{P_0}} & \underline{\Hom}_A(P_1,N)\otimes R \ar[d]^-{\Psi_{P_1}}\\
0\ar[r] & \underline{\Hom}_A(M,N\otimes R)\ar[r] & \underline{\Hom}_A(P_0,N\otimes R)\ar[r] & \underline{\Hom}_A(P_1,N\otimes R) 
}
\]
in the category of $H$-comodules with exact rows.  By part (3), the maps in the rows are graded $H^{\op}$-colinear. Here, $\Psi_{P_0},\Psi_{P_1}$ are isomorphisms since $P_0$ and $P_1$ are finite free $A$-modules. By the 5-lemma, $\Psi_M$ is also an isomorphism. This concludes the proof. 
\end{proof}

\begin{remark}
    Suppose $M$ and $N$ are right graded relative $(A, H)$-modules and \eqref{eq:coactionHom} is replaced by 
    \[\rho(f)(m)=\sum f(m_0)_0\otimes f(m_0)_1S(m_1).\] Then by \cite[Lemma 2.2]{Ul90} a right graded relative module version of \Cref{lem:relHom} holds without the hypothesis that the antipode $S$ is bijective. 
\end{remark}

\subsection{Numerical regularity}
\label{subsec:numerical regularity}

In this subsection, we study the homological invariants of Tor-regularity (Torreg) and Castelnuovo--Mumford regularity (CMreg). CMreg is a classical notion in algebraic geometry and in commutative algebra that measures when the associated Hilbert function becomes a polynomial. In \cite{Romer}, R\"{o}mer expands this notion to a numerical measure of CM-regularity and $\Tor$-regularity. These were extended to the noncommutative setting by J\o rgensen \cite{Jorgensen1999, Jorgensen2004}, and Dong and Wu \cite{DongWu2009}. For a graded algebra $A$, $\mathrm{Torreg}({\kk})=0$ if and only if $A$ is Koszul. Hence, one can think of $\mathrm{Tor}$-regularity as a measure of Koszulity of an algebra (e.g., see \cite[Example 2.4(4) and Lemma 5.6]{Kirkman-Won-Zhang2021}). In \cite{Kirkman-Won-Zhang2021}, Kirkman, Won, and Zhang introduced numerical AS-regularity (ASreg) as a measure of AS-regularity of a noncommutative algebra \cite{Kirkman-Won-Zhang2021}. In the case that $A$ is a noetherian connected graded algebra with balanced dualizing complex, the numerical AS-regularity is zero if and only if $A$ is AS-regular in the classical sense \cite[Theorem 0.8]{Kirkman-Won-Zhang2021}.
We show that Torreg and CMreg are Morita--Takeuchi invariant. As a consequence, (numerical) ASreg is invariant under Morita--Takeuchi equivalence. 

AS-regular algebras are viewed as noncommutative analogues of polynomial rings. Below, we review the classical definition of AS-regularity. 

\begin{defn}[{\cite{ArtinSchelter1987} and \cite[Definition 2.11]{ATV1990}}]
\label{defn:AS-regular}
A connected graded algebra $A$ is called \emph{Artin--Schelter regular} (or AS-regular) of dimension $d$ if the following conditions hold.
\begin{itemize}
    \item[(AS1)] $A$ has finite global dimension $d$, and
    \item[(AS2)] $A$ is {\it Gorenstein}, that is, for some integer $l$,
\[
\underline{\Ext}_A^i(\kk,A)=\begin{cases}
\kk(l) & \text{if}\ i=d\\
0 & \text{if}\ i\neq d,
\end{cases}
\]
where $\kk$ is the trivial module $A/A_{\ge 1}$. We call the above integer $l$ the \emph{AS index of $A$}.
\end{itemize} 
\end{defn}

\begin{remark}
We note that \Cref{defn:AS-regular} as given in \cite{ArtinSchelter1987,ATV1990} also assumes that $A$ has polynomial growth or equivalently, that $A$ has finite Gelfand--Kirillov (GK) dimension. However, polynomial growth is not invariant under the Morita--Takeuchi equivalence. Hence, we follow the conventions of e.g. \cite{Mori-Smith2016, vdb2017, Zhang1998} to explore a larger family of noncommutative algebras.
\end{remark}

In the following, we give the definition of various measures of regularity. Suppose $A$ is a graded algebra, $M$ a graded right $A$-module, and $N$ a graded left $A$-module. For each $i$, the group $\Tor_i^A(M,N)$ has a graded structure, which we denote by $\underline{\Tor}_i^A(M,N)$. 

\begin{defn}[{\cite{DongWu2009, Jorgensen1999, Jorgensen2004,Kirkman-Won-Zhang2021}}] \label{def:Torreg}
Let $A$ be a connected graded algebra with graded Jacobson radical $\mathfrak m:=A_{\ge 1}$. Let $M$ be a finitely generated graded left $A$-module.
\begin{itemize}
    \item[(1)] The \emph{Tor-regularity} of $M$ is 
    \[ \mathrm{Torreg}(M) = \sup_{i,j \in \mathbb Z} \left\{ j-i : \underline{\mathrm{Tor}}^A_i(\kk,M)_j \neq 0 \right\}. \]
    \item[(2)] The \emph{Castelnuovo--Mumford regularity (or CM-regularity)} of $M$ is 
    \[ \mathrm{CMreg}(M) = \sup_{i,j \in \mathbb Z} \left\{ j+i : H_\mathfrak m^i(M)_j \neq 0 \right\},\]
    where $H_\mathfrak m^i(M) = \lim\limits_{n\to \infty}\underline{\Ext}_A^i(A/\mathfrak m^n, M)$ is the \emph{$i$th local cohomology} of $M$.
    \item[(3)] The {(numerical) \it Artin--Schelter regularity (or AS-regularity)} of $A$ is
    \[ {\rm ASreg(A)}={\rm Torreg}(\kk)+{\rm CMreg}(A). \]
     \item[(4)] We call $M$ \emph{$s$-Cohen--Macaulay} or simply \emph{Cohen--Macaulay} if $H_\mathfrak m^i(M)=0$ for all $i\neq s$ and $H_\mathfrak m^s(M)\neq 0$. We say $A$ is \emph{Cohen--Macaulay} if $\!_AA$ is Cohen--Macaulay.
    \end{itemize}
\end{defn}

\begin{remark}
\label{noeth-homol-reg}
\begin{itemize}
\item[(a)] The definitions of homological regularities given in \cite{Kirkman-Won-Zhang2021} all have the additional assumption that $A$ is noetherian. Since general Morita--Takeuchi equivalence does not preserve noetherianity, for our results we do not require the noetherian assumption.

\item[(b)] We use \cite[Definition 1.2]{Kirkman-Won-Zhang2021} for the above notion of $s$-Cohen--Macaulay and note that there are several notions of Cohen--Macaulay in the literature (e.g., \cite[Theorem 4.8 (i)-(iii)]{BrownMacleod} and \cite[Definition 5.8]{Le}).

\item[(c)] The original definitions of $\mathrm{Tor}$-regularity and CM-regularity in \cite{DongWu2009, Jorgensen1999, Jorgensen2004} were given for complexes of graded modules. These notions coincide with the above when identifying a graded left $A$-module $M$ with the complex $M^\bullet$ with 
   \[ M^\bullet_i = \begin{cases} M & i = 0 \\ 0 & i \neq 0 \end{cases} \] and zero morphisms. 
\end{itemize}
\end{remark}

We show below that all these notions of regularity are invariant under Morita--Takeuchi equivalence. The following lemma is a crucial step towards establishing this invariance but first, we provide the context that will be used for the rest of this section. Let $H$ and $K$ be Hopf algebras that are Morita--Takeuchi equivalent via the monoidal equivalence $(\widetilde{(-)},\ \xi^{-1}): {\rm comod}(H) \stackrel{\otimes}{\cong} {\rm comod}(K)$ described in \eqref{eq:MTE}. For a connected graded $H$-comodule algebra $A$, $\widetilde{A}$ is the corresponding connected graded $K$-comodule algebra. By restriction, we have an induced equivalence between the two categories of graded relative modules 
\begin{align}\label{functor-rel}
\!_{\Gr(A)}{\mathcal M}^H~ &\cong~ \!_{\Gr(\widetilde{A})}{\mathcal M}^{K}\\
M~&\mapsto ~\widetilde{M}\notag
\end{align}
for any $M\in \!_{\Gr(A)}{\mathcal M}^H$. 

Now, we recall the tensor product of two graded relative modules in the sense of \cite[Definition 7.8.21]{EGNO}. View the connected graded $H$-comodule algebra $A$ as a connected graded algebra in the monoidal category ${\rm comod}(H)$ as in \cite[Definition 2.4.8]{EGNO}. Let $M$ be a right graded relative $(A,H)$-module and $N$ be a left graded relative $(A,H)$-module. By \cite[Definition 7.8.5]{EGNO}, $M$ and $N$ are graded modules over the graded algebra $A$ in ${\rm comod}(H)$. Then a \emph{tensor product} of $M$ and $N$ over $A$ is the object $M\otimes_AN\in {\rm comod}(H)$ defined as the co-equalizer of the diagram \begin{align}\label{def:tensorproduct}
\xymatrix{
M\otimes A\otimes N\ar@<-0.5ex>[rr]_-{\mu_M\otimes \id_N}\ar@<0.5ex>[rr]^-{\id_M\otimes \mu_N} && M\otimes N\ar[r]\ar@{=} &M\otimes_AN
}
\end{align}
where $\mu_M: M\otimes A\to M$ and $\mu_N: A\otimes N\to N$ are the graded $A$-module structure maps on $M$ and $N$, respectively. It is clear that $M\otimes_AN$ is a graded $H$-comodule and it coincides with the usual tensor product over $A$ by applying the forgetful functor from ${\rm comod}(H)$ to the category of vector spaces.

\begin{lemma}
\label{tensor-twist}
Let $M$ be a right graded relative $(A,H)$-module and $N$ be a left graded relative $(A,H)$-module. Then there is a natural isomorphism  
\[
\widetilde{M \otimes_A N}~ \cong~ \widetilde{M} \otimes_{\widetilde{A}} \widetilde{N},
\]
of graded $K$-comodules.
\end{lemma}

\begin{proof}
Applying the monoidal equivalence $(\widetilde{(-)}, \xi^{-1})$ between ${\rm comod}(H)$ and ${\rm comod}(K)$ to \eqref{def:tensorproduct}, we have the commutative diagram
\[
\xymatrix{
\widetilde{M\otimes A\otimes N}\ar@<-0.5ex>[rr]_-{\widetilde{\mu_M\otimes \id_N}}\ar@<0.5ex>[rr]^-{\widetilde{\id_M\otimes \mu_N}}\ar[d]^{\cong} && \widetilde{M\otimes N}\ar[r]\ar[d]^{\cong} &\widetilde{M\otimes_AN} \ar[r]\ar[d]  & 0 \\
\widetilde{M}\otimes \widetilde{A}\otimes \widetilde{N}\ar@<-0.5ex>[rr]_-{\mu_{\widetilde{M}}\otimes \id_{\widetilde{N}}}\ar@<0.5ex>[rr]^-{\id_{\widetilde{M}}\otimes \mu_{\widetilde{N}}} && \widetilde{M}\otimes \widetilde{N}\ar[r] & \widetilde{M}\otimes_{\widetilde{A}} \widetilde{N}  \ar[r] & 0 
}
\]
with exact rows in ${\rm comod}(K)$. Our result follows from the 5-lemma.   
\end{proof}

\begin{lemma}
\label{Cor:Tor2cocycle}
Let $M$ be a right graded relative $(A,H)$-module and $N$ be a left graded relative $(A,H)$-module. Then we have an isomorphism 
\[ \widetilde{\underline{\Tor}^A_i(M,N)}~\cong~ \underline{\Tor}^{\widetilde{A}}_i(\widetilde{M},\widetilde{N}) \]
of graded $K$-comodules, for all $i\ge 0$.
\end{lemma}

\begin{proof}
By \Cref{free-res-rel}, we have a resolution $P_\bullet$ of $N$ in $\,_{\Gr(A)}{\mathcal M}^H$ consisting of free graded relative $(A,H)$-modules. It is clear that the equivalence in \eqref{functor-rel} between relative modules induced by the monoidal equivalence $(\widetilde{(-)},\xi^{-1})$ sends free graded relative $(A,H)$-modules to free graded relative $(\widetilde{A},K)$-modules. Thus, $\widetilde{P_\bullet}$ is a free graded resolution of $\widetilde{N}$ in $\,_{\Gr(\widetilde{A})}{\mathcal M}^K$. As a consequence, both $P_\bullet$ and $\widetilde{P_\bullet}$ are free graded resolutions of $N$ and $\widetilde{N}$ as graded modules over $A$ and $\widetilde{A}$, respectively, so \Cref{tensor-twist} implies that
\begin{align*}
\underline{\Tor}^{\widetilde{A}}_i(\widetilde{M},\widetilde{N}) ={\rm H}_i \left(\widetilde{M}\otimes_{\widetilde{A}} \widetilde{P_\bullet}\right)
\cong {\rm H}_i \left(\widetilde{M\otimes_A P_\bullet}\right)
\cong \widetilde{{\rm H}_i (M\otimes_A P_\bullet)} 
=\widetilde{\underline{\Tor}^A_i(M,N)}.
\end{align*}
For the last isomorphism, we use the fact that the monoidal equivalence $(\widetilde{(-)},\xi^{-1})$ is exact and so it commutes with the homology functor. 
\end{proof}

When $H$ is finite-dimensional, the following results were observed in \cite[Corollary 4.5, Proposition 4.7]{CS2}. We now prove it for an \emph{arbitrary} Hopf algebra $H$ using $\Tor$. For background on Koszulity and $N$-Koszulity, see \cite{Berger2001,PolishchukPositselski2005}. 

\begin{Cor}
\label{Cor:gldim}
Let $A$ and $\widetilde{A}$ be as above. We have the following Morita--Takeuchi invariants for connected graded comodule algebras:  
\begin{itemize}
    \item[(1)] ${\rm gl.dim}(A)={\rm gl.dim}(\widetilde{A})$;
    \item[(2)] $A$ is $N$-Koszul if and only if $\widetilde{A}$ is $N$-Koszul. 
\end{itemize}
\end{Cor}

\begin{proof}
By \Cref{rem:trivial}, the trivial module $\kk=A/A_{\ge 1}$ is a relative $(A,H)$-module. Since $\widetilde{A}$ is again connected graded, $\kk=\widetilde{A}/\widetilde{A}_{\ge 1}$ is also a relative $(\widetilde{A},K)$-module. 

(1): By \cite[Comment before Lemma 2.3]{SteZha} and \Cref{Cor:Tor2cocycle}, we have 
   \[{\rm gl.dim}(A)=\max\{i\,|\, \Tor^A_i(\kk,\kk)\neq 0\}=\max\{i\,|\, \widetilde{\Tor^A_i(\kk,\kk)}\neq 0\}=\max\{i\,|\, \Tor^{\widetilde{A}}_i(\kk,\kk)\neq 0\}={\rm gl.dim}(\widetilde{A}).\]

(2): A connected graded algebra $A$ is $N$-Koszul if each term in the graded minimal resolution of $\kk$ is concentrated in just one degree. This fact is equivalent to $\underline{\Tor}^A_i(\kk,\kk)$ being concentrated in just one degree for all $i$ and so, the result follows by \Cref{Cor:Tor2cocycle}. 
\end{proof}

Recall the monoidal equivalence \eqref{functor-rel} $\widetilde{(-)}: M\mapsto \widetilde{M}:=M\square_HT=(M\otimes T)^{\co H}$, where $T$ is an $H$-$K$-bi-Galois object. Note that in the expression $M\square_HT$ we treat $T$ as a left $H$-comodule, while in the expression $(M\otimes T)^{\co H}$ we treat $T$ as a right $H$-comodule via the antipode of $H$. The algebra structure on $T$ remains the same in both cases. The following lemma is an analogue of \cite[Proposition 3.2]{CS2} in the case of relative modules.

\begin{lemma}
\label{lem:moduleMTE}
Let $H,K$ and $A$ be as above. Suppose $H$ has a bijective antipode. Let $M,N\in \!_{{\rm Gr}(A)}\mathcal M^{H}$, and assume $M$ has a presentation of finite free graded relative $(A,H)$-modules in $\,_{{\rm Gr}(A)}\mathcal M^{H}$. Then there is a natural isomorphism of graded $K$-comodules 
\[
\widetilde{\underline{\Hom}_A(M,N)}~\cong~\underline{\Hom}_{\widetilde{A}}(\widetilde{M},\widetilde{N}).
\]
\end{lemma}

\begin{proof}
Suppose $P_1\to P_0\to M\to 0$ is a presentation of $M$ in $\,_{{\rm Gr}(A)}\mathcal M^{H}$, where $P_0,P_1$ are finite free graded $(A,H)$-modules. Since $\widetilde{(-)}$ is a monoidal equivalence, $\widetilde{P_1}\to \widetilde{P_0}\to \widetilde{M}\to 0$ is a presentation of $\widetilde{M}$ in $\,_{{\rm Gr}(\widetilde{A})}\mathcal M^{K}$, where $\widetilde{P_0},\widetilde{P_1}$ are also finite free graded $(\widetilde{A},K)$-modules. By \Cref{lem:relHom}(1), $\underline{\Hom}_A(M,N)$ has a graded right $H$-comodule structure and $\underline{\Hom}_{\widetilde{A}}(\widetilde{M},\widetilde{N})$ has a graded right $K$-comodule structure. 

Since $T$ is an $H$-$K$-bi-Galois object, $T$ is injective as a left $H$-comodule. Therefore, we can apply \cite[Proposition 3.1 and Proposition 3.3]{CS2} to obtain the following isomorphisms of graded $K$-comodules
\begin{align*}
 \underline{\Hom}_{\widetilde{A}}(\widetilde{M},\widetilde{N})\cong \underline{\Hom}_A^H(M,N\otimes T)\cong  \underline{\Hom}_A(M,N\otimes T)^{\co H}\cong \left(\underline{\Hom}_A(M,N)\otimes T\right)^{\co H}\cong \widetilde{\underline{\Hom}_A(M,N)},
\end{align*}
where the second and third isomorphisms follow from \Cref{lem:relHom}(2) and from \Cref{lem:relHom}(4), respectively. 
\end{proof}

The following result is a generalization of \cite[Theorem 4.2]{CS2}, to the extent that $H$ is not assumed to be finite-dimensional. 

\begin{thm}
\label{thm:tensor-twist}
Let $H,K$ and $A$ be as above. Suppose $H$ has a bijective antipode, and $A$ is noetherian. Let $M,N$ be two graded relative $(A,H)$-modules, where $M$ is finitely generated over $A$. Then for each $i \geq 0$, there is an isomorphism 
\[ \widetilde{\underline{\Ext}_A^i(M,N)}~ \cong~ \underline{\Ext}_{\widetilde{A}}^i(\widetilde{M}, \widetilde{N}) \]
of graded $K$-comodules.
\end{thm}

\begin{proof}
Since $A$ is noetherian and $M$ is finitely generated over $A$, by \Cref{free-res-rel}, we have a resolution $P_\bullet$ of $M$ in $\,_{\Gr(A)}{\mathcal M}^H$ consisting of finite free graded relative $(A,H)$-modules. It follows that $\widetilde{P_\bullet}$ is a finite free graded resolution of $\widetilde{M}$ in $\,_{\Gr(\widetilde{A})}{\mathcal M}^K$. In particular, both $P_\bullet$ and $\widetilde{P_\bullet}$ are finite free graded resolutions of $M$ and $\widetilde{M}$ as graded modules over $A$ and $\widetilde{A}$, respectively. Hence we can apply \Cref{lem:moduleMTE} to obtain the following isomorphisms of graded $K$-comodules: 
\begin{align*}
\underline{\Ext}_{\widetilde{A}}^i(\widetilde{M},\widetilde{N})&={\rm H}^i\left(\underline{\Hom}_{\widetilde{A}}(\widetilde{P}_\bullet,\widetilde{N})\right)
\cong 
{\rm H}^i\left(\widetilde{\underline{\Hom}_A(P_\bullet,N)}\right) 
\cong \widetilde{{\rm H}^i (\underline{\Hom}_A(P_\bullet},N)) \cong \widetilde{\underline{\Ext}_A^i(M,N)},
\end{align*}
where the second to last isomorphism follows from the fact that $\widetilde{(-)}$ is an exact functor. 
\end{proof}

We now prove that certain homological properties, including numerical regularities, for connected graded comodule algebras are invariant under Morita--Takeuchi equivalence. 

\begin{thm}
\label{thm:Torreg}
Let $H$ be any Hopf algebra with a bijective antipode and $A$ be a noetherian connected graded $H$-comodule algebra. The following numerical regularities are invariant under any Morita--Takeuchi equivalence:
 \begin{itemize}
 \item[(1)] ${\rm Torreg}(\kk)$, where $\kk$ is considered as a trivial module over $A$;
 \item[(2)] ${\rm CMreg}(A)$; and
 \item[(3)] ${\rm ASreg}(A)$.
\end{itemize}
Additionally, if $A$ is either
\begin{itemize}
 \item[(4)] $s$-Cohen--Macaulay, or 
 \item[(5)] AS-regular of dimension $d$,
 \end{itemize}
then so is its image under Morita--Takeuchi equivalence. 
\end{thm}

\begin{proof}
(4): Note that the $i$th local cohomology of $A$ can be computed as the direct limit of the direct system $\{\underline{\Ext}_A^i(A/A_{\ge n},A)\,|\, n\ge 0\}$ using a resolution of $A/A_{\ge n}$ consisting of finite free graded relative $(A,H)$-modules. It is clear that $\widetilde{A/A_{\ge n}}=\widetilde{A}/\widetilde{A}_{\ge n}$ for any $n\ge 0$ as graded relative $(\widetilde{A},K)$-modules. Hence \Cref{thm:tensor-twist} implies that the two directed systems 
\[ 
\left\{\widetilde{\underline{\Ext}_A^i(A/A_{\ge n},A)} \,\, \middle|\, n\ge 0\right\}\qquad \text{ and }\qquad \left\{\underline{\Ext}_{\widetilde{A}}^i(\widetilde{A}/\widetilde{A}_{\ge n},\widetilde{A})\, \middle|\, n\ge 0\right\} \] 
are naturally isomorphic as graded $K$-comodules. By \Cref{def:Torreg}, the statement follows.

(5): By \Cref{Cor:gldim}, we have ${\rm gl.dim}(A)={\rm gl.dim}(\widetilde{A})$. \Cref{thm:tensor-twist} implies that $\widetilde{\Ext_A^i(\kk,A)}\cong \Ext_{\widetilde{A}}^{i}(\kk,\widetilde{A})$ for every $i$ since both $\kk=A/A_{\ge 1}$ and $A$ are relative $(A,H)$-modules. The statement follows by \Cref{defn:AS-regular}.

(1): This follows from \Cref{Cor:Tor2cocycle} since both $\kk=A/A_{\ge 1}$ and $A$ are relative $(A,H)$-modules. 

(2): This follows from the same argument as (4).

(3): By definition of ${\rm ASreg}(A)$ in \Cref{def:Torreg}, this follows directly from parts (1) and (2). 
\end{proof}

\subsection{Frobenius property}
\label{subsec:Frobenius}

In this subsection, we show that the Frobenius property is preserved under Morita--Takeuchi equivalence. Recall that a finite-dimensional algebra $A$ is said to be \emph{Frobenius} if there exists an associative nondegenerate bilinear form, called Frobenius form,  $f: A \times A \to \kk$ such that $f(ab,c)=f(a,bc)$, for any $a,b,c \in A$ (see e.g., \cite[\S8]{CMZ} and \cite[\S2]{KJ} for equivalent definitions). While Frobenius algebras are known to have many applications, we are interested in proving its invariance for use in applications to AS-regular algebras, due to the fact that a connected graded algebra $R$ is AS-regular if and only if its Ext-algebra is Frobenius \cite{LPWZ}. We now consider the \emph{Nakayama automorphism} of a Frobenius algebra $A$, which is an automorphism $\mu_A$ of $A$ satisfying 
\begin{equation}\label{eq:Naka}
f(a,b)=f(b,\mu_A(a)),
\end{equation}
for any $a,b\in A$, where $f: A\times A\to \kk$ is the associative nondegenerate bilinear form mentioned above. Every Frobenius algebra admits a Nakayama automorphism \cite[Lemma 3.3]{Smith1994}. In particular, $f$ induces an isomorphism of $A$-bimodules such that $\,^{\mu_A}A \cong A^*$ via $a\mapsto f(a,-)$ for any $a\in A$, where $\,^{\mu_A}A$ is the $A$-bimodule whose left action is twisted by $\mu_A$. As a consequence, the Nakayama automorphism $\mu_A$ is uniquely determined up to an inner automorphism of $A$.

\begin{thm}
\label{lem:twistF}
Let $H$ be any Hopf algebra and $A$ be a finite-dimensional right $H$-comodule algebra. Suppose $A$ is Frobenius with an associative nondegenerate form induced by some $H$-colinear map $f: A\otimes A\to D$ for some one-dimensional $H$-comodule $D$.  Then the Frobenius property of $A$ is preserved under Morita--Takeuchi equivalence.  
\end{thm}

\begin{proof}
Denote the graded multiplication of $A$ by $m: A\otimes A\to A$, which is $H$-colinear. We apply the monoidal equivalence $\widetilde{(-)}: {\rm comod}_{\rm fd}(H) \to {\rm comod}_{\rm fd}(K)$ in \eqref{eq:MTE} to obtain the corresponding finite-dimensional $K$-comodule algebra $\widetilde{A}$ with multiplication given by 
\[
n~:=~\left( \widetilde{A} \otimes \widetilde{A}\xrightarrow{\xi^{-1}_{A,A}}\widetilde{A\otimes A}\xrightarrow{\widetilde{m}} \widetilde{A}\right). 
\]  
We define a bilinear form on $\widetilde{A}$, which is induced by the following composition of $K$-colinear maps 
$$
g~:=~\left(\widetilde{A}\otimes \widetilde{A}\xrightarrow{\xi^{-1}_{A,A}} \widetilde{A\otimes A}\xrightarrow{\widetilde{f}} \widetilde{D}\right),
$$
where $\widetilde{D}$ is again a one-dimensional $K$-comodule by \cite[Proposition 1.16]{Bichon2014}.
One can check that $g$ is associative with respect to the multiplication $n$ of $\widetilde{A}$ due to the commutative diagram 
\[
\xymatrix{
\widetilde{A}\otimes \widetilde{A}\otimes \widetilde{A}\ar[rrr]^-{{\rm id}\otimes n}\ar[ddd]_-{n\otimes {\rm id}} \ar[dr]_-{\xi^{-1}_{A^{\otimes 3}}} &&& \widetilde{A}\otimes \widetilde{A} \ar[ddd]^{g} \ar[dl]^-{\xi^{-1}_{A^{\otimes 2}}}\\
&\widetilde{A\otimes A\otimes A} 
\ar[r]^-{\widetilde{{\rm id}\otimes m}}\ar[d]_-{\widetilde{m\otimes {\rm id}}}&\widetilde{A\otimes A}\ar[d]^-{\widetilde{f}} &\\
&\widetilde{A\otimes A} \ar[r]^-{\widetilde{f}}& \widetilde{D}\ar@{=}[dr]&\\
\widetilde{A}\otimes\widetilde{A} \ar[ur]_-{\xi^{-1}_{A^{\otimes 2}}} \ar[rrr]^-{g}&&& \widetilde{D}.
}
\]

Denote by $D^*$ the dual object of $D$ in ${\rm comod}_{\rm fd}(H)$ (which is a tensor-inverse to $D$, since $D$ is 1-dimensional). Note that in the sense of \cite[\S 1.5.1]{TV}, since $f: A\otimes A\to D$ is a nondegenerate bilinear map, it induces a nondegenerate pairing between $D^*\otimes A$ and $A$ via the following $H$-comodule maps 
\[
\alpha:~=~\left((D^*\otimes A)\otimes A\xrightarrow{\cong} D^*\otimes (A\otimes A)\xrightarrow{\id \otimes f} D^*\otimes D\xrightarrow{\cong} \kk\right).
\]
The copairing of $\alpha$ is given by 
\[
\beta~:=~\left(\kk\xrightarrow{{\rm coev}} A\otimes A^*\xrightarrow{\id \otimes \eta^{-1}} A\otimes (D^*\otimes A)\right),
\]
where 
\begin{align*}
\eta:&=\left(D^*\otimes A\xrightarrow{\id \otimes {\rm coev}} (D^*\otimes A)\otimes (A\otimes A^*)\xrightarrow{\cong} (D^*\otimes (A\otimes A))\otimes A^*\right.\\
& \qquad \left.\xrightarrow{(\id \otimes f)\otimes \id} (D^*\otimes D)\otimes A^* \xrightarrow{\cong} \kk\otimes A^*=A^*\right)
\end{align*}
is an isomorphism of $H$-comodules because $f$ is nondegenerate. Here, $A^*$ denotes the left dual of $A$ in ${\rm comod}_{\rm fd}(H)$. Hence by \cite[Lemma 1.5]{TV}, 
\[
\tau~:=~\left((\widetilde{D}^*\otimes \widetilde{A})\otimes \widetilde{A}\xrightarrow{\cong}\widetilde{D}^*\otimes (\widetilde{A}\otimes \widetilde{A}) \xrightarrow{\id \otimes \,g} \widetilde{D}^*\otimes\widetilde{D} \xrightarrow{\cong} \kk\right)
\]
is again a nondegenerate pairing,  with inverse 
\begin{align*}
\mu~:=~&\left(\kk=\widetilde{\kk}\xrightarrow{\widetilde{\beta}} \widetilde{A\otimes (D^*\otimes  A)}\xrightarrow{\xi_{A,D^*\otimes A}} \widetilde{A}\otimes \widetilde{D^*\otimes A}\xrightarrow{\id \otimes \, \xi_{D^*,A}} \widetilde{A}\otimes (\widetilde{(D^*)}\otimes \widetilde{A}) \xrightarrow{\cong} \widetilde{A}\otimes (\widetilde{D}^* \otimes \widetilde{A})\right),
\end{align*}
where we identify $\widetilde{(D^*)}=(\widetilde{D})^*$ as one-dimensional $K$-comodules. In particular, the following identities hold 
\begin{align*}
(\tau\otimes \id)\circ (\id\otimes \mu)=\id_{\widetilde{D}^*\otimes \widetilde{A}}\quad{\text{and}}\quad (\id\otimes \tau)\circ (\mu \otimes \id)=\id_{\widetilde{A}}.
\end{align*}
By applying the forgetful functor to the category of finite-dimensional vector spaces, we can identify $\widetilde{D}^*=\widetilde{D}=\kk$ and $\widetilde{D}^*\otimes \widetilde{A}=\widetilde{A}$. Therefore, we have $f=\alpha$ and $g=\tau$ as linear maps. In this case, we can write $\mu(1)=\sum v_i\otimes v^i$ with $v_i\in \widetilde{A}$ and $v^i\in \widetilde{D}^*\otimes \widetilde{A}=\widetilde{A}$. Then the above identities imply that 
\[
\sum g(a,v_i)\,v^i~=~a~=~\sum v_i\,g(v^i,a),
\]
for any $a\in \widetilde{A}$. Since we identify $g$ and $\tau$ as linear maps and $\tau$ is a nondegenerate pairing, so is $g$. As a consequence, $\widetilde{A}$ is Frobenius.  \end{proof}

Suppose for an integer $d \geq 0$,  $A=\bigoplus_{0 \leq i\leq d} A_i$ is a connected graded Frobenius algebra. By \cite[Lemma 3.2]{Smith1994}, the multiplication of $A$
 defines a nondegenerate bilinear form
 \[
 A_i\times A_{d-i}\to A_d
 \]
for each $0\leq i\leq d$. In particular, we have $\dim A_d=\dim A_0=1$. In this case, we can always choose the Nakayama automorphism $\mu_A$ to be a graded automorphism of $A$, which is uniquely determined by $A$. This motivates the following corollary.

\begin{Cor}
\label{cor:Frobconn}
Let $A$ be a finite-dimensional connected graded $H$-comodule algebra. Then the Frobenius property of $A$ is preserved under Morita--Takeuchi equivalence.   
\end{Cor}

\section{Applications of 2-cocycle twists}
\label{sec:applications}

In this last section, we study some applications of 2-cocycle twists, particularly for AS-regular algebras. Throughout, let $H$ be any Hopf algebra and $\sigma$ be a right 2-cocycle on $H$. We employ the monoidal equivalence 
\begin{align*}
  F : \mathrm{comod}(H)~&\overset{\otimes}{\cong}~ \mathrm{comod}(H^\sigma)  \\
  U~&\mapsto~F(U)=:U_\sigma\notag
\end{align*}
from \eqref{eq:equiv2}, between right comodule categories over $H$ and $H^\sigma$, respectively. We write $\otimes$ and $\otimes_\sigma$ for the tensor products in the corresponding right comodule categories. Explicitly, since $H = H^\sigma$ as coalgebras, $F$ is given by the identity functor on objects, together with an isomorphism of $H^\sigma$-comodules as in \eqref{eq:monoidalstructure}
\begin{align*} 
\xi_{U,V}: F(U \otimes V )~&\xrightarrow{\sim}~ F(U) \otimes_\sigma F(V)\\
u\otimes v~&\mapsto~ \sum \sigma^{-1}(u_1,v_1)\,u_0\otimes v_0,\notag
\end{align*} 
where the following diagram 
\begin{align}\label{eq:cdsigma}
\xymatrix{
F(U\otimes V\otimes W) \ar[rr]^-{\xi_{U\otimes V,W}}\ar[d]_-{\xi_{U, V\otimes W}} && F(U \otimes V) \otimes_\sigma F(W) \ar[d]^-{\xi_{U,V}\otimes \id}\\
F(U) \otimes_\sigma F(V \otimes W) \ar[rr]^-{\id \otimes \xi_{V,W}} && F(U) \otimes_\sigma F(V) \otimes_\sigma F(W) 
}
\end{align}
commutes for any right $H$-comodules $U,V,W$. Note that $F$ sends a connected graded $H$-comodule algebra $A$ to the twisted connected graded $H^\sigma$-comodule algebra $F(A)=A_\sigma$. Since $F$ is the identity functor, for any $H$-comodule $V$, we may identify $F(V^{\otimes m}) = F(V)^{\otimes_\sigma m} = V^{\otimes m}$ as vector spaces.

\subsection{Twisting of generalized superpotential algebras}
\label{subsec:superpotential}

Superpotential algebras and their associated universal quantum groups have attracted significant interest in recent years (see e.g., \cite{BDV13, Chirvasitu-Walton-Wang2019, Dubois-Violette2005, DVM, DVL1990, WaltonWang2016}). We are particularly interested in superpotential algebras since all $N$-Koszul AS-regular algebras admit superpotentials \cite[Theorem 11]{DVM}. In this subsection, we consider a generalization of superpotential algebras, namely, derivation quotient algebras, and determine their behavior under 2-cocycle twists arising from coactions of arbitrary Hopf algebras.

Suppose $V$ is a finite-dimensional right $H$-comodule. The vector space dual $V^*={\rm Hom}_\kk(V,\kk)$ is again a right $H$-comodule by using the antipode map $S$ of $H$. We can view $V^*$ as the left dual of $V$ in the category ${\rm comod}(H)$ of right $H$-comodules, where the canonical evaluation map ${\rm ev}: V^*\otimes V\to \kk$ and coevaluation map ${\rm coev}: \kk \to V\otimes V^*$ are both $H$-comodule maps.

\begin{defn}
\label{defn:derivation}
Let $N$ and $m$ be integers with $2\leq N\leq m$. Let $V$ be a finite-dimensional right comodule over a Hopf algebra $H$ with a subcomodule $j: W \hookrightarrow V^{\otimes m}$. For any integer $1\leq i\leq m$, we define 
\begin{itemize}
    \item[(1)] subcomodules $\partial^iW$ of $V^{\otimes (m-i)}$ such that 
    \begin{align*} { \partial^1W = {\rm Im}\left(V^*\otimes W\xrightarrow{\id \otimes j} V^*\otimes (V^{\otimes m})\xrightarrow{\cong } (V^*\otimes V)\otimes V^{\otimes (m-1)} \xrightarrow{{\rm ev}\otimes \id} V^{\otimes (m-1)}\right) } \end{align*} 
    and $\partial^{i+1}W=\partial(\partial ^iW) $, inductively, and  
    \item[(2)] the $(m-N)$-th order \emph{derivation-quotient algebra} is given by \[A(W,N):=TV/(\partial^{m-N}W),\] where $TV$ is the tensor algebra of $V$.
\end{itemize}
\end{defn}

When $H=\kk$, \Cref{defn:derivation} yields the same derivation-quotient algebras discussed in \cite[\S 1.5.1]{Mori-Smith2016}. In particular, when $H=\kk$ and $W$ is one-dimensional, $A(W,N) $ is the superpotential algebra discussed in \cite{Dubois-Violette2005}. 
Next, we describe how $A(W,N)$ behaves under twisting by a $2$-cocycle $\sigma$ on $H$. When $\sigma$ is derived from a graded automorphism of $A(W,N)$ as in \Cref{Introthm:2cycleZhang}, we recover results of \cite{Mori-Smith2016} in terms of Zhang twists.            

\begin{lemma}
Let $2\leq N\leq m$ be integers. Let $V$ be a finite-dimensional right $H$-comodule with a subcomodule $W\subseteq V^{\otimes m}$. Then $A(W, N)$ is a right $H$-comodule algebra induced by the right $H$-comodule structure on $V$.
\end{lemma}

\begin{proof}
By \Cref{defn:derivation} and the discussion above, we know $\partial^{m-N}(W)\subseteq V^{\otimes N}$ is an $H$-subcomodule. It is then straightforward to check that $A(W, N)=TV/(\partial^{m-N}W)$, as a quotient algebra of $TV$, is again a right $H$-comodule algebra. 
\end{proof}

\begin{defn}
\label{def:twistform}
Retain the above notation.  
\begin{itemize}
    \item[(1)] For any $v\in V^{\otimes m}$, we define a new element 
\begin{align*}
v_{\sigma}:=\left((\id_{V^{\otimes (m-2)}}\otimes \xi_{V,V})\circ\cdots\circ(\id_V\otimes \xi_{V, V^{\otimes(m-2)}})\circ \xi_{V,V^{\otimes (m-1)}}\right)(v)
\end{align*}
in $V^{\otimes m}$; and
\item[(2)] for any subspace $W\subseteq V^{\otimes m}$, we define a new subspace 
\begin{align*}
W_{\sigma}:&=\left((\id_{V^{\otimes (m-2)}}\otimes \xi_{V,V})\circ\cdots\circ(\id_V\otimes \xi_{V, V^{\otimes(m-2)}})\circ \xi_{V,V^{\otimes (m-1)}}\right)(W)\notag\\
&=\{w_{\sigma}\,|\, w\in W\}
\end{align*}
of $V^{\otimes m}$.
\end{itemize}
\end{defn}

\begin{remark}
Let $R$ be a subspace of $V^{\otimes m}$ and consider the connected graded algebra $A=TV/(R)$ subject to $m$-homogeneous relations. For any graded automorphism $\phi$ of $A$, by \Cref{lem:2cocycleGrade}, the Zhang twist $A^\phi$ can be realized as a 2-cocycle twist $A_\sigma$, where $\sigma$ is a right 2-cocycle on Manin's universal quantum group $\underline{\rm aut}^r(A)$ determined by a twisting pair. Moreover, we can show that the natural isomorphism $$\xi_{V^{\otimes i},V^{\otimes j}}: F(V^{\otimes i}\otimes V^{\otimes j})\to F(V^{\otimes i})\otimes_\sigma F(V^{\otimes j})$$ in \eqref{eq:monoidalstructure} is given by $\xi_{V^{\otimes i},V^{\otimes j}}=(\id_V)^{\otimes i}\otimes (\phi^{-1})^{\otimes j}$ for any $i,j\ge 0$. As a consequence, the twisted element of any $v\in V^{\otimes m}$ defined in \Cref{def:twistform}(1) is given by 
\[ v_\sigma= \left(\id_V\otimes (\phi^{-1}) \cdots (\phi^{-1})^{\otimes (m-1)}\right)(v).
\]
\end{remark}

As a consequence of the commutative diagram \eqref{eq:cdsigma}, there exists a unique $H^{\sigma}$-comodule isomorphism $F(V^{\otimes m})\to F(V)^{\otimes_\sigma m}$ given by all possible compositions of identity maps, $\xi_{V^{\otimes i},V^{\otimes j}}$, $\id_{V^{\otimes(m-i-j)}}$ and parentheses, which we denote 
\begin{align}\label{eq:coherence}
\xi_{V^{\otimes m}}:F(V^{\otimes m})\to F(V)^{\otimes_\sigma m}.
\end{align}
By abuse of notation, we view $\xi_{V^{\otimes m}}$ as an automorphism on the vector space $V^{\otimes m}$. Thus we can write $v_{\sigma}=\xi_{V^{\otimes m}}(v)$ and $W_{\sigma}=\xi_{V^{\otimes m}}(W)$ as in \Cref{def:twistform}. 
In view of \Cref{lem:2cocycleGrade}, the following result extends \cite[Lemma 5.1]{Mori-Smith2016} from Zhang twists to 2-cocycle twists.

\begin{lemma}
\label{lem:twistR}
Let $V$ be a finite-dimensional right $H$-comodule, $R$ be a subcomodule of $V^{\otimes m}$ for some $m\ge 2$, and $A=TV/(R)$. Then for any right $2$-cocycle $\sigma$ on $H$, we have $A_{\sigma}\cong TV/(R_{\sigma})$ as graded right $H^\sigma$-comodule algebras.
\end{lemma}

\begin{proof}
We apply the monoidal equivalence functor $F: {\rm comod}(H)\xrightarrow{\sim} {\rm comod}(H^\sigma)$ described in \eqref{eq:equiv2} to the $H$-comodule embedding $j:R \hookrightarrow  V^{\otimes m}$. Thus we have 
\[
R_{\sigma}={\rm Im}\left(F(R)\xrightarrow{F(j)} F(V^{\otimes m})\xrightarrow{\xi_{V^{\otimes m}}} F(V)^{\otimes_\sigma m}\right).
\]
Therefore, $R_{\sigma}\subseteq F(V)^{\otimes_\sigma m}$ is an $H^{\sigma}$-subcomodule and hence $TV/(R_{\sigma})$ is a graded right $H^\sigma$-comodule algebra.
We have a graded $H^\sigma$-comodule algebra map $f: TV \to \left(TV/(R)\right)_{\sigma}$ satisfying $f|_V=\id$, since $TV$ is free. Note that $f|_{V^{\otimes \ell}}=\xi^{-1}_{V^{\otimes \ell}}$. Hence, 
\[
f(R_\sigma)=\xi^{-1}_{V^{\otimes \ell}}(R_\sigma)=\xi^{-1}_{V^{\otimes \ell}}\circ \xi_{V^{\otimes \ell}}(R)=R,
\]
which is $0$ in $\left(TV/(R)\right)_{\sigma}$, and so $f$ factors through $TV/(R_\sigma)$, which we still denote as $f: TV/(R_\sigma)\to \left(TV/(R)\right)_{\sigma}$.

Similarly, since $\sigma^{-1}$ is a right 2-cocycle of $H^{\sigma}$, we have a graded $H$-comodule algebra map $g: TV/(R) \to \left(TV/(R_\sigma)\right)_{\sigma^{-1}}$ satisfying $g|_V=\id$. It induces a graded $H^\sigma$-comodule algebra map 
\[
g_{\sigma}: (TV/(R))_{\sigma}\xrightarrow{g} \left((TV/(R_\sigma))_{\sigma^{-1}}\right)_{\sigma}=TV/(R_\sigma)
\]
which still satisfies $g_{\sigma}|_V=\id$. Now it is routine to check that $f$ and $g_{\sigma}$ are inverse to each other. 
\end{proof}

Our next result is a 2-cocycle version of \cite[Proposition 5.2]{Mori-Smith2016}, which was proven for Zhang twists.

\begin{proposition}
\label{prop:2.2.5}
Let $2\leq N\leq m$ be integers. Let $V$ be a finite-dimensional right $H$-comodule, $j:W \hookrightarrow V^{\otimes m}$ be an $H$-subcomodule, and $\sigma$ be a right 2-cocycle on $H$. Then
\begin{itemize}
    \item[(1)] $\partial^i(W_{\sigma})=(\partial^i W)_{\sigma}$;
    \item[(2)] If $A(W,N)=TV/(R)$, where $R\subseteq V^{\otimes N}$ is some $H$-subcomodule, then $A(W_{\sigma},N)=TV/(R_{\sigma})$;
    \item[(3)] $A(W_{\sigma},N)\cong A(W,N)_{\sigma}$.
\end{itemize}
\end{proposition}

\begin{proof}
(1): One checks that 
\begin{align*}
\partial^i(W_{\sigma}) &= {\rm Im}\left((F(V)^*)^{\otimes_\sigma i}\otimes_{\sigma} F(W)\xrightarrow{\id \otimes_{\sigma} F(j)}(F(V)^*)^{\otimes_\sigma i}\otimes_{\sigma} F(V^{\otimes m})\right. \\
&\qquad \qquad \xrightarrow{\id \otimes_{\sigma} \xi_{V^{\otimes m}}}(F(V)^*)^{\otimes_\sigma i}\otimes_{\sigma} F(V)^{\otimes_\sigma m}\\
&\qquad \qquad \xrightarrow{\cong }((F(V)^*)^{\otimes_\sigma i}\otimes_{\sigma} F(V)^{\otimes_\sigma i})\otimes_{\sigma} F(V)^{\otimes_{\sigma} (m-i)}\left.\xrightarrow{{\rm ev}\otimes \id} F(V)^{\otimes_\sigma (m-i)} \right)   
\end{align*}
and 
\begin{align*}
    (\partial^i W)_{\sigma} &= {\rm Im}\left(F((V^*)^{\otimes i}\otimes W)\xrightarrow{F(\id \otimes j)}F((V^*)^{\otimes i}\otimes V^{\otimes m})\right. \\
    &\qquad \qquad \xrightarrow{F(\cong)}F(((V^*)^{\otimes i}\otimes V^{\otimes i})\otimes V^{\otimes (m-i)})\xrightarrow{F({\rm ev}\otimes \id)}F(V^{\otimes (m-i)})\\
    &\qquad \qquad \left.\xrightarrow{\xi_{V^{\otimes (m-i)}}}F(V)^{\otimes_\sigma (m-i)}\right).
\end{align*}
Note that there is a natural isomorphism $\iota:(F(V)^*)^{\otimes_\sigma i} \cong F(V^*)^{\otimes_\sigma i}\cong F((V^*)^{\otimes i})$. Then the two above images are equal by the following commutative diagram:
\[
\xymatrix{ 
(F(V)^*)^{\otimes_\sigma  i} \otimes_\sigma F(W) \ar[r]^-{\iota\otimes\id} \ar[d]_-{\id \otimes_{\sigma}\, F(j)} &F\left( (V^*)^{\otimes i}\right) \otimes_\sigma F(W)\ar[d]^-{\xi^{-1}_{(V^*)^{\otimes i},W}}\\
(F(V)^*)^{\otimes_\sigma i} \otimes_\sigma F(V^{\otimes m }) \ar[d]_-{\id \otimes_\sigma\, \xi_{V^{\otimes m}}}  & F\left( (V^*)^{\otimes i} \otimes W \right) \ar[d]^-{F(\id \otimes j)} \\
F(V^*)^{\otimes_\sigma i} \otimes_\sigma F(V)^{\otimes_\sigma m }\ar[d]_-{\cong } & F\left( (V^*)^{\otimes i} \otimes V^{\otimes m} \right) \ar[d]^-{F(\cong)}\\
\left((F(V)^*)^{\otimes_\sigma i} \otimes_\sigma F(V)^{\otimes_\sigma i }\right)\otimes_\sigma F(V)^{\otimes_\sigma (m-i)}\ar[d]_-{{\rm ev} \otimes_{\sigma}\, \id} & F\left(((V^*)^{\otimes i} \otimes V^{\otimes i}) \otimes V^{\otimes (m-i)}\right)\ar[d]^-{F({\rm ev}\otimes \id)} \\
F(V)^{\otimes_\sigma (m-i)} & F(V^{\otimes (m-i)}). \ar[l]^-{\xi_{V^{\otimes (m-i)}}}
}
\]

(2): If $A(W,N)=TV/(R)$, then $R=\partial^{m-N}W$ by definition. Thus by (1), we have
\[
A(W_\sigma,N)=TV/(\partial^{m-N}W_\sigma)=TV/((\partial^{m-N}W)_\sigma)=TV/(R_\sigma).
\] 

(3): Denote by $A(W,N)=TV/(R)$, where $R=\partial^{m-N}W$. It follows from (2) and \Cref{lem:twistR} that
\[
A(W_\sigma,N)=TV/(R_\sigma)\cong(TV/(R))_{\sigma}=A(W,N)_{\sigma}. \qedhere
\]
\end{proof}

Let $V={\rm Span}_\kk(v_1,\ldots,v_n)$ be a right $H$-comodule via $\rho: v_i\mapsto \sum_{j=1}^nv_j\otimes h_{ji}$ for some $h_{ji}\in H$. In the following result, we will view $V^*={\rm Span}_\kk(v^1,\ldots,v^n)$ as a left $H$-comodule via 
\begin{align}\label{RLcoaction}
\rho^*: v^i \mapsto \sum_{j=1}^n h_{ij}\otimes v^j.
\end{align}
For two right $H$-comodules $U$ and $V$, we define the invertible linear map 
\begin{align}\label{def:adjointF}
    \zeta_{U^*,V^*}: U^*\otimes V^*&\to U^*\otimes V^* \\
    f\otimes g&\mapsto \sum \sigma^{-1}(f_{-1},g_{-1})\, f_0g_0, \notag
\end{align}
for any $f\in U^*$ and $g\in V^*$, with inverse $\zeta^{-1}_{U^*,V^*}: f\otimes g\mapsto \sum \sigma(f_{-1},g_{-1})f_0g_0$, where $U^*$ and $V^*$ are viewed as left $H$-comodules as above. We denote by 
\begin{align}\label{eq:coherenceE}
\zeta_{(V^*)^{\otimes m}}: (V^*)^{\otimes m}\to (V^*)^{\otimes m}
\end{align}
the unique invertible linear map given by all possible compositions of  $\id_{{V^*}^{\otimes(m-i-j)}}$, $\zeta_{(V^*)^{\otimes i},(V^*)^{\otimes j}}$, and parentheses. 

Let $m\ge 2$ be any integer and $R\subseteq V^{\otimes m}$ be a right $H$-subcomodule. The connected graded algebra $A=TV/(R)$ is a right $H$-comodule algebra with the induced right $H$-coaction on $V$.  We denote  by $R^\perp=\{f\in (V^*)^{\otimes m}\,|\, f(r)=0 \text{ for all } 
 r\in R \}$ and the dual algebra by $A^!=TV^*/(R^\perp)$. By a higher degree version of \cite[Proposition 2.5]{CWZ2014} (also see \Cref{lemma:bialgebraM}), $A^!$ is a left $H$-comodule algebra with the left $H$-coaction induced by \eqref{RLcoaction}. In particular, $R^\perp\subseteq (V^*)^{\otimes m}$ is a left $H$-subcomodule. Recall that for any right 2-cocycle $\sigma$ on $H$, $\sigma^{-1}$ is a left 2-cocycle on $H$ and one can define the twisted algebra $\,_{\sigma^{-1}}(A^!)$ on the left.

\begin{proposition}
\label{lem:dualtwist}
Retain the above notation. Then we have
\[
\left(A_{\sigma}\right)^!~\cong  \,_{\sigma^{-1}}(A^!).  
\]
as left $H^\sigma$-comodule algebras. 
\end{proposition}
\begin{proof}
In view of \Cref{lem:twistR}, we have
\[
\left(A_{\sigma}\right)^!=\left((TV/(R))_{\sigma}\right)^!\cong \left(TV/(R_\sigma)\right)^!=TV^*/((R_\sigma)^\perp).
\]
By the canonical pairing \[\langle-,-\rangle:
(V^*)^{\otimes m} \times V^{\otimes m}\to \kk,\] 
one checks that $\zeta_{(V^*)^{\otimes m}}=\left(\xi_{V^{\otimes m}}\right)^*$. Note that $(R_\sigma)^\perp$ is defined as $\langle (R_\sigma)^\perp,R_\sigma\rangle=0$, which implies that \begin{align*}
0 &= \langle (R_\sigma)^\perp,\,R_\sigma\rangle\\
\overset{\eqref{eq:coherence}}&{=} \langle (R_\sigma)^\perp,\, \xi_{V^{\otimes m}}(R)\rangle \\
&= 
\langle (\xi_{V^{\otimes m}})^*(R_\sigma)^\perp,\,  R\rangle\\
&= \langle  \zeta_{(V^*)^{\otimes m}}(R_\sigma)^\perp,\,  R\rangle.
\end{align*}
Thus, $(R_\sigma)^\perp=(\zeta_{(V^*)^{\otimes m}})^{-1}\,(R^\perp) = \zeta^{-1}_{(V^*)^{\otimes m}}\,(R^\perp)$. Then we get
\[
\left(A_{\sigma}\right)^! \cong TV^*/((R_\sigma)^\perp) = TV^*/(\zeta^{-1}_{(V^*)^{\otimes m}}\,(R^\perp)) \cong \,_{\sigma^{-1}}  (TV^*/(R^\perp))= \,_{\sigma^{-1}} (A^!),
\] 
where the last isomorphism follows by a left-sided version of \Cref{lem:twistR}.
\end{proof}

\subsection{Twisting of AS-regular algebras}
\label{subsec:twist AS-reg}

In this subsection, we study 2-cocycle twists of $N$-Koszul AS-regular algebras that are comodule algebras over a Hopf algebra $H$ and show that 2-cocycle twists preserve the AS-regular property. This allows one to remove the noetherianity assumption necessary in \Cref{thm:Torreg}. When $H$ is finite-dimensional, a similar result was proved in \cite[Lemma 4.12]{CKWZ} (for semisimple) and in \cite[Theorem 4.8]{CS2}. Here, we mainly focus on the $H$-coactions when $H$ is infinite-dimensional. 

Let $H$ be any Hopf algebra, $A=\bigoplus_{i\ge 0}A_i$ be a connected graded right $H$-comodule algebra, and $\sigma$ be any right 2-cocycle on $H$. We denote by $A_\sigma=\bigoplus_{i\ge 0}(A_i)_\sigma$ the twisted algebra, which is a connected graded $H^\sigma$-comodule algebra via the monoidal equivalence \eqref{eq:equiv2}.

\begin{thm}
\label{thm:AS}
Let $A$ be an $N$-Koszul and AS-regular algebra of dimension $d$. Let $H$ be any Hopf algebra that right coacts on $A$ preserving the grading of $A$. Then for any right 2-cocycle $\sigma$ on $H$, the twisted algebra $A_{\sigma}$ is still $N$-Koszul and AS-regular of the same dimension $d$.
\end{thm}

\begin{proof}
By \Cref{Cor:gldim}(2), $B:=A_{\sigma}$ is $N$-Koszul. As a vector space, set 
\[
\mathscr{E}(A):=\bigoplus_{} (A^!)_{\delta(k)},
\]
where $\delta(k)=km/2$ if $k$ is even and $\delta(k)=(k-1)m/2+1$ if $k$ is odd. Define a multiplication on $\mathscr{E}(A)$ by
\[
f*g:=\begin{cases}
    (-1)^{ij} f \cdot g, & N=2\\
    f \cdot g, & N>2, \text{ either } i \text{ or } j \text{ even}\\
    0, & \text{else}
\end{cases}
\]
for $f$ and $g$ homogeneous of degrees $i$ and $j$, respectively. Here $f \cdot g$ refers to the multiplication of $f$ and $g$ in $A^!$. Likewise, define $\mathscr{E}(B)$. By \cite[Proposition 3.1]{BM06} and \cite[Proposition 2.3]{He-Lu2005}, the Ext-algebra $E(A)$ of $A$ is isomorphic to $\mathscr{E}(A)$, and likewise for $B$. By \cite[Corollary D]{LPWZ}, a connected graded algebra is AS-regular if and only if its $\Ext$-algebra is Frobenius. Therefore, it suffices to show that $E(B)$ is Frobenius whenever $E(A)$ is Frobenius. Since $A$ is a right $H$-comodule algebra, $A^!$ is a left $H$-comodule algebra with the left $H$-coaction induced by \eqref{RLcoaction}. Thus, $E(A)$ is a left $H$-comodule algebra. Moreover, using the definition of multiplication on $\mathscr{E}(A)$ and $\mathscr{E}(B)$ as above, one may check that \Cref{lem:dualtwist} implies that $E(B)\cong \!_{\sigma^{-1}} E(A)$ as algebras. Thus, by \Cref{cor:Frobconn}, $E(B)$ is Frobenius as desired.
\end{proof}

\subsection{Nakayama automorphisms under 2-cocycle twists}
\label{subsec:Nakayama}

In this subsection, we describe the Nakayama automorphism of an $N$-Koszul AS-regular algebra $A$ under 2-cocycle twists. We first establish results for general Frobenius algebras, which we will apply to the case of the Frobenius Ext-algebra of $A$. 

Let $d \geq 0$ be an integer and $E=\bigoplus_{0 \leq m \leq d} \ E_m$ be a connected graded Frobenius algebra. Following the language of \cite[\S 3]{CWZ2014}, we describe the Nakayama automorphism $\mu_E$ in terms of a basis of each homogeneous piece $E_m$.  Since $E$ is Frobenius, the multiplication of $E$ yields a nondegenerate bilinear form
 \[ 
E_m \times E_{d-m}\to E_d
 \]
for each $0\leq m \leq d$ \cite[Lemma 3.2]{Smith1994}. In particular, $\dim E_d=\dim E_0=1$. Let $e$ be a nonzero element in $E_d$. Pick any basis $\{a_1,\ldots,a_n\}$ of an arbitrary homogeneous piece $E_{m}$. There is a unique basis of $E_{d-m}$, say $\{b_1,\ldots,b_n\}$, such that
\begin{equation}\label{eq:4:aibj}
a_ib_j = \delta_{ij}e
\end{equation}
for all $1\leq i,j\leq n$. Moreover, there is a unique new basis of $E_m$, say $\{c_1,\ldots,c_n\}$, such that
\begin{equation}\label{eq:4:bicj}
b_ic_j = \delta_{ij}e
\end{equation}
for all $1\leq i,j\leq n$. Therefore, there is a nonsingular matrix $(\alpha_{ij})_{1 \leq i,j \leq n}$ such that
\[
c_i = \sum_{j=1}^n \alpha_{ij}a_j
\]
for all $1\leq i\leq n$. Recall that the Nakayama automorphism $\mu_E$ is an automorphism of $E$ which satisfies \eqref{eq:Naka} 
\[
f(a,b)=f(b,\mu_E(a)),
\]
for any $a,b\in E$, where $f: E\times E\to \kk$ is an associative nondegenerate bilinear form. By this definition, when restricted to $E_m$, $\mu_E$ maps $a_i$ to $c_i$, 
\begin{equation} \label{eq:4:alpha}
\mu_E(a_i) = c_i= \sum_{j=1}^n\alpha_{ij}a_j,
\end{equation}
for all $1\leq i\leq n$. 

Now let $H$ be a Hopf algebra coacting on $E$ from the left via $\rho: E \to H\otimes E$, while preserving the grading of $E$. We use $\{y_{ij}\}_{1\leq i,j\leq n}$ to denote the set of elements in $H$ such that the left $H$-coaction on the homogeneous piece $E_m$ is given by
\begin{equation} \label{eq:4:yis}
\rho(a_i)=\sum_{s = 1}^n  y_{is} \otimes a_s,
\end{equation} 
for all $i$. From the comodule axioms, one can check that $\Delta(y_{ij})=\sum_{s=1}^n y_{is}\otimes y_{sj}$ and $\varepsilon(y_{ij})=\delta_{ij}$. For simplicity, we denote the matrix $(y_{ij})_{n\times n}$ by $\mathbb Y$. Similarly, we write 
\begin{equation} \label{eq:4:fis-gis}
\rho(b_i)=\sum_{s=1}^n f_{is}\otimes b_s, \qquad \text{ and } \qquad 
\rho(c_i)=\sum_{s=1}^n g_{is}\otimes c_s,
\end{equation} 
for elements $f_{is},g_{is}\in H$. Let $\mathbb F=(f_{ij})_{n\times n}$ and $\mathbb G=(g_{ij})_{n\times n}$ be the corresponding matrices. 

The following result was originally stated in \cite[Lemma 3.1]{CWZ2014} for the degree-one piece $E_1$ of $E$. Here, we restate it for any degree-$m$ piece $E_m$ of $E$ since a similar argument works for all homogeneous pieces. For any $n \times n$ matrix $\mathbb V=(v_{ij})$, we denote by $\mathbb V^T$ the transpose matrix of $\mathbb V$. When we apply the antipode map $S$ of $H$ on $\mathbb V$ entry-wise, we write $S(\mathbb V)=(S(v_{ij}))_{n \times n}$.

\begin{lemma}[{\cite[Lemma 3.1]{CWZ2014}}]\label{lem:coationIde}
Retain the above notation and let $\mathbb I$ be the $n\times n$ identity matrix. Let $g \in H$ be the homological codeterminant of the left $H$-coaction on $E$ such that $\rho(e)=g\otimes e$. We have the following identities:
\begin{itemize}
    \item[(1)] $\mathbb YS(\mathbb Y)=\mathbb I=S(\mathbb Y)\mathbb Y$;
    \item[(2)] $\mathbb GS(\mathbb G)=\mathbb I=S(\mathbb G)\mathbb G$;
    \item[(3)] $\mathbb Y\, \mathbb F^T=g\mathbb I$, and, as a consequence, $S(\mathbb Y)=\mathbb F^T(g^{-1}\mathbb I)$;
    \item[(4)] $S(\mathbb F)S(\mathbb Y^T)=g^{-1}\mathbb I$.
    \item[(5)] $S(\mathbb G)S(\mathbb F^T)=g^{-1}\mathbb I$, and, as a consequence, $S(\mathbb F^T)\cdot (g\mathbb I)=\mathbb G$;
    \item[(6)] $S^2(\mathbb Y)=(g\mathbb I)\cdot \mathbb G\cdot (g^{-1}\mathbb I)$;
    \item[(7)] $\mathbb G=\alpha \mathbb Y\alpha^{-1}$, where $\alpha=(\alpha_{ij})$.
\end{itemize}
 \end{lemma}

Recall that the Frobenius property is invariant under any Morita--Takeuchi equivalence (\Cref{lem:twistF}), in particular, under any 2-cocycle twists. We now provide a formula that shows how the Nakayama automorphism of a Frobenius algebra changes under a 2-cocycle twist. Here, we consider a left cocycle twist of $E$ by a left 2-cocycle. We will apply this result in \Cref{Cor:Naka} to describe the Nakayama automorphism of a twisted AS-regular algebra.

\begin{thm}
\label{lem:Nakatwist}
Retain the above notation and let $g$ be the homological codeterminant of the left $H$-coaction $\rho$ on $E$ such that $\rho(e)=g\otimes e$. For any left 2-cocycle $\sigma$ on $H$, the Nakayama automorphism $\mu_{\!_\sigma E}$ of the twisted Frobenius algebra $\,_\sigma E$ is given by $\mu_{\!_\sigma E}(a_i)=\sum_{1\leq j\leq n}\widetilde{\alpha_{ij}}a_j$, where  
\[
\widetilde{\alpha_{ij}}=\sum_{1\leq p,q,r,s,t\leq n}\sigma\left(y_{ip},S(y_{pq})g\right)\,\sigma\left(g^{-1}S^2(y_{st}),g\right)\,\sigma^{-1}\left(g^{-1}S^2(y_{rs}),S(y_{qr})g\right)\,\alpha_{tj}, \]
and $\mu_E(a_i) = \sum_{j=1}^n\alpha_{ij}a_j$, for all $1 \leq i,j \leq n$. 
In particular, if the $H$-coaction on $E$ has trivial homological codeterminant, then 
\[
\widetilde{\alpha_{ij}}=\sum_{1\leq p,q,r,s\leq n}\sigma\left(y_{ip},S(y_{pq})\right)\,\sigma^{-1}\left(S^2(y_{rs}),S(y_{qr})\right)\,\alpha_{sj}. \]
\end{thm}
  
\begin{proof}
We write $\cdot_\sigma$ for the twisted product in $\,_\sigma E$ such that $a\cdot_\sigma b=\sum \sigma(a_{-1},b_{-1})a_0b_0$ for any $a,b\in E$. For any homogeneous degree $1 \leq m \leq d$, pick a basis $\{a_1,\ldots,a_n\}$ for $E_m$. As in the above discussion, let $\{b_1, \ldots, b_n\}$ and $\{c_1, \ldots, c_n\}$ be bases for $E_{d-m}$ and $E_m$ respectively.
Denote the new bases for $\,_\sigma{E}_{d-m} = E_{d-m}$ and $\,_\sigma{E}_{m} = E_m$ by $\{\widetilde{b_1},\ldots, \widetilde{b_n}\}$ and $\{\widetilde{c_1},\ldots, \widetilde{c_n}\}$, respectively,  such that under the twisted product
\begin{equation}
\label{eq:4:aisigmabj}
a_i\cdot_\sigma \widetilde{b_j}=\delta_{ij}e \qquad \text{and} \qquad \widetilde{b_i}\cdot_\sigma \widetilde{c_j}=\delta_{ij}e,
\end{equation} 
for all possible $i,j$. Then, there are invertible matrices $\mathbb M=(m_{ij})_{1 \leq i,j \leq n},\mathbb N=(n_{ij})_{1 \leq i,j \leq n}$ such that
\begin{equation} \label{eq:Thm4.3.2:M}
\widetilde{b_j}=\sum_{k=1}^n m_{jk}b_k \qquad \text{and} \qquad \widetilde{c_j}=\sum_{k=1}^n n_{jk}c_k,
\end{equation} 
A straightforward calculation yields that
\begin{align*}
\delta_{ij}e \overset{(\ref{eq:4:aisigmabj})}&= a_i\cdot_\sigma \widetilde{b_j} \\
\overset{(\ref{eq:Thm4.3.2:M})}&=\sum_{1\leq k\leq n} m_{jk} a_i\cdot_\sigma b_k\\
\overset{(\ref{eq:4:yis})\text{ and }(\ref{eq:4:fis-gis})}&=\sum_{1\leq k,s,t\leq n} m_{jk} \sigma(y_{is},f_{kt})a_sb_t\\
\overset{(\ref{eq:4:aibj})}&=\sum_{1\leq k,s\leq n} m_{jk} \sigma(y_{is},f_{ks})e\\
&=\sum_{1\leq k,s\leq n} m_{jk} \sigma(y_{is},S(y_{sk})g)e,
\end{align*}
where the last equality comes from \Cref{lem:coationIde}(3). Similarly, we have
\begin{align*}
\delta_{ij}e \overset{(\ref{eq:4:aisigmabj})}&= \widetilde{b_i} \cdot_\sigma \widetilde{c_j} \\
\overset{(\ref{eq:Thm4.3.2:M})}&=\sum_{1\leq k,l\leq n} m_{ik}n_{jl} b_k\cdot_\sigma c_l\\
\overset{(\ref{eq:4:fis-gis})}&=\sum_{1\leq k,l,s,t\leq n} m_{ik}n_{jl} \sigma(f_{ks},g_{lt}) b_sc_t\\
\overset{(\ref{eq:4:bicj})}&=\sum_{1\leq k,l,s\leq n} m_{ik}n_{jl} \sigma(f_{ks},g_{ls})e\\
&=\sum_{1\leq k,l,s\leq n} m_{ik}n_{jl} \sigma(S(y_{sk})g,g^{-1}S^2(y_{ls})g)e,
\end{align*}
where the last equality comes from \Cref{lem:coationIde}(3),(6). Hence, we have a system of equations 
\begin{align}\label{eq:2cocyle1}
  \sum_{1\leq k,s\leq n} m_{jk} \sigma(y_{is},S(y_{sk})g) = \delta_{ij} = \sum_{1\leq k,l,s\leq n} m_{ik}n_{jl} \sigma(S(y_{sk})g,g^{-1}S^2(y_{ls})g)
\end{align}
for all possible $i,j$. 
In order to solve for $\mathbb M$ and $\mathbb N$ in terms of $\sigma$, we need the following identity
\begin{align}\label{eq:2cocycle}
     \sum \sigma(x_1,S(x_2)g)\sigma(S(x_3),g)\sigma^{-1}(S(x_4),x_5)&=\sum \sigma(x_1,S(x_4))\sigma(x_2S(x_3),g)\sigma^{-1}(S(x_5),x_6)\\
     &=\sum \sigma(x_1,S(x_2))\sigma^{-1}(S(x_3),x_4)\notag\\
     &=\varepsilon(x)\notag
\end{align}
for any $x\in H$, where the first equality follows from \Cref{defn:cocycle} of a left 2-cocycle, the second equality follows from the structural property of the antipode map $S$ and the facts that $\sigma$ is normal and $\varepsilon(g)=1$, and the last equality is from \cite[Theorem 1.6(a5)]{Doi93}. Let $x=S(y_{tk})g$ in \eqref{eq:2cocycle} above, we get
 \begin{align}\label{eq:2cocyle3}
\sum_{1\leq s,l,q,p\leq n}\sigma(S(y_{sk})g,g^{-1}S^2(y_{ls})g) \sigma(g^{-1}S^2(y_{ql}),g)\sigma^{-1}(g^{-1}S^2(y_{pq}),S(y_{tp})g)=\varepsilon(S(y_{tk})g)= \delta_{tk}.
\end{align}
For simplicity, we write matrices $\mathbb A=(a_{lk})_{1 \leq l,k \leq n}$, $\mathbb B=(b_{tl})_{1 \leq t,l \leq n}$ and $\mathbb C=(c_{ik})_{1 \leq i,k \leq n}$ such that 
\begin{align*}
a_{lk}&= \sum_{1\leq s\leq n} \sigma(S(y_{sk})g,g^{-1}S^2(y_{ls})g), \\
b_{tl}&=\sum_{1\leq q,p\leq n}\sigma(g^{-1}S^2(y_{ql}),g)\sigma^{-1}(g^{-1}S^2(y_{pq}),S(y_{tp})g),\\
c_{ik}&=\sum_{1\leq s\leq n}\sigma(y_{is},S(y_{sk})g).
\end{align*}
Equations \eqref{eq:2cocyle1} and \eqref{eq:2cocyle3} can be written in matrix forms as follows:
 \[
\mathbb N\mathbb A\mathbb M^T=\mathbb I,\quad \mathbb C\mathbb M^T=\mathbb I,\quad  \mathbb B\mathbb A=\mathbb I,
 \]
which implies that $\mathbb N=\mathbb C\mathbb B$. By definition of the Nakayama automorphism,  when restricted to $(\,_\sigma E)_m = E_m$, we have
\begin{align*}
     \mu_{\!_\sigma E}(a_i) = \widetilde{c_i} \overset{(\ref{eq:Thm4.3.2:M})}= \sum_{1\leq t\leq n} n_{it}c_t \overset{(\ref{eq:4:alpha})}= \sum_{1\leq t,j\leq n} n_{it} \alpha_{tj} a_j \overset{\mathbb N = \mathbb C \mathbb B}= \sum_{1\leq q,t,j\leq n} c_{iq} b_{qt} \alpha_{tj} a_j = \sum_{1\leq j\leq n} \widetilde{\alpha_{ij}} a_j
\end{align*}
and as a consequence, $\widetilde{\alpha_{ij}} = \sum_{1\leq q,t\leq n} c_{iq}b_{qt}\alpha_{tj}$. In particular, if the $H$-coaction on $E$ has trivial homological codeterminant, then $g=1$. Our last assertion follows immediately.
\end{proof}

Let $A$ be an AS-regular algebra of dimension $d$ and $H$ be any Hopf algebra with bijective antipode that right coacts on $A$ preserving the grading of $A$. We pick a basis $\{x_1,\ldots,x_n\}$ of $A_1$ and write the right $H$ coaction on $A_1$ as 
\[
\rho(x_i)=\sum_{j=1}^n x_j\otimes y_{ji},
\]
for some $y_{ji}\in H$ satisfying $\Delta(y_{ji})=\sum_{s=1}^n y_{js}\otimes y_{si}$ and $\varepsilon(y_{ji})=\delta_{ji}$. Now suppose $\sigma$ is any right 2-cocycle on $H$, consider the twisted connected graded algebra $A_\sigma$. By \Cref{thm:AS}, $A_\sigma$ is again AS-regular of dimension $d$. Our next result links the Nakayama automorphisms $\mu_A$ of $A$ and $\mu_{A_\sigma}$ of $A_\sigma$. We write
\[
   \mu_A(x_i)=\sum_{j=1}^n \alpha_{ji}x_j,\quad \text{and}\quad \mu_{A_\sigma}(x_i)=\sum_{j=1}^n \widetilde{\alpha_{ji}}x_j
\]
for some invertible matrices $(\alpha_{ji})_{1 \leq j,i \leq n}$ and $(\widetilde{\alpha_{ji}})_{1 \leq j,i \leq n}$.

\begin{Cor}
\label{Cor:Naka}
Retain the above notation and let $g$ be the homological codeterminant of the right $H$-coaction on $A$. Suppose $A$ is $N$-Koszul. For any right 2-cocycle $\sigma$ on $H$, the Nakayama automorphism $\mu_{A_\sigma}$ of the twisted AS-regular algebra $A_\sigma$ is given by $\mu_{A_\sigma}(x_i)=\sum_{1\leq j\leq n}\widetilde{\alpha_{ji}}x_j$, where 
\[
\widetilde{\alpha_{ji}}=\sum_{1\leq p,q,r,s,t\leq n}\sigma^{-1}\left(y_{jp},S(y_{pq})g\right)\,\sigma^{-1}\left(g^{-1}S^2(y_{st}),g\right)\,\sigma\left(g^{-1}S^2(y_{rs}),S(y_{qr})g\right)\,\alpha_{ti}, \]
and $\mu_A(x_i)=\sum_{j=1}^n \alpha_{ji}x_j$, for all $1 \leq i,j \leq n$.
\end{Cor}

\begin{proof}
By \Cref{thm:AS} and \Cref{Cor:gldim}(2), both $A$ and $A_\sigma$ are $N$-Koszul AS-regular algebras. We can apply \cite[Theorem 6.3]{BM06} to obtain the corresponding Nakayama automorhpisms 
\[
\mu_{E(A)}(x_i^*)=\sum_{j=1}^n (-1)^{d+1}\alpha_{ij}x_j^*,\quad \text{and}\quad \mu_{E(A_\sigma)}(x_i^*)=\sum_{j=1}^n (-1)^{d+1}\widetilde{\alpha_{ij}}x_j^*,
\]
of the Frobenius Ext-algebras $E(A)$ and $E(A_\sigma)$ of $A$ and of $A_\sigma$, respectively. Following the proof of \Cref{thm:AS}, $E(A)$ is a graded left $H$-comodule algebra and $E(A_\sigma) \cong \!_{\sigma^{-1}}E(A)$.
Applying \Cref{lem:Nakatwist} to $\,_{\sigma^{-1}}E(A)$ and the duality relation between $E(A_\sigma)$ and $A_\sigma$, we obtain the desired expression for $\mu_{A_\sigma}$. 
\end{proof}

The following homological identity was first established in \cite[Theorem 0.3]{RRZ1} to describe how the Nakayama automorphism of an AS-regular algebra changes under a (left) Zhang twist (also see \cite[Theorem 5.5]{Mori-Smith2016} for the special $N$-Koszul case). Here, we provide another proof using our earlier result, that a (right) Zhang twist can be realized as a 2-cocycle twist via Manin's universal quantum group. 

\begin{Cor}
Let $A$ be a $N$-Koszul AS-regular algebra with AS index $l$. For any automorphism $\phi$ of $A$, the Nakayama automorphism of the twisted $N$-Koszul AS-regular algebra $A^\phi$ is given by
\[
\mu_{A^\phi}=\mu_A\circ \phi^{-l}\circ \xi_{\hdet(\phi)},
\]
where $\hdet(\phi)$ is the homological determinant of $\phi$ and the graded algebra automorphism $\xi_{\hdet(\phi)}$ of $A$ is given by $\xi_{\hdet(\phi)}(a)=\hdet(\phi)^{|a|}a$, for any homogeneous element $a\in A$.
\end{Cor}

\begin{proof}
By \cite[Theorem 11]{DVM}, there is an $n$-dimensional $\kk$-vector space $V$ such that $A$ is a superpotential algebra \[A=A(N,W)=TV/(\partial ^{l-N}(W)),\]
for some one-dimensional subspace $W\subset V^{\otimes l}$ corresponding to a preregular form $f: (V^*)^{\otimes l}\to \kk$. We use a presentation of Manin's universal quantum group $\underline{\rm aut}^r(A)$ described in \cite[Definition 5.17]{Chirvasitu-Walton-Wang2019}. We pick a basis $\{x_1,\ldots,x_n\}$ for $V$ and a dual basis $\{x^1,\ldots,x^n\}$ for $V^*$ and write $f_{i_1\cdots i_l}=f(x^{i_1},\ldots,x^{i_l})\in \kk$. The graded algebra $\underline{\rm aut}^r(A)$ is generated by $\mathbb Y=(y_{ij})_{1 \leq i,j \leq n},\mathbb Z = (z_{ij})_{1 \leq i,j \leq n}$ and $g^{\pm 1}$ subject to 
\begin{equation*}
\left. 
\begin{aligned}
    \sum_{1\leq i_1,\ldots,i_l\leq n}{f}_{i_1 \cdots i_l}y_{j_1i_1}\cdots y_{j_li_l} &= {f}_{j_1\cdots j_l}g^{-1}, &\textnormal{ for any } 1\leq j_1,\dots, j_l\leq n,\\
    \sum_{1\leq i_1,\ldots,i_l\leq n}{f}_{i_1\cdots i_l}z_{j_li_l}\cdots z_{j_1i_1} &= {f}_{j_1\cdots j_l}g,  &\textnormal{ for any } 1\leq j_1,\dots, j_l\leq n,\\
    gg^{-1} &= g^{-1}g=1, \text{ and } \\
    \mathbb{Y}\mathbb{Z} &= \mathbb{I}_{n\times n},
\end{aligned}\right\}
\end{equation*}
where $\deg(y_{ij})=1,\deg(z_{ij})=-1$, and $\deg(g^{\pm 1})=\mp l$. For the Hopf algebra structure of $\underline{\rm aut}^r(A)$, we have $\Delta(z_{ij})=\sum_{k=1}^n z_{kj}\otimes z_{ik}$ for all possible $i,j$, the antipode $S(\mathbb Y)=\mathbb Z$ and $S^2$ is a graded automorphism \cite[Proposition 5.23, Corollary 5.27]{Chirvasitu-Walton-Wang2019}. The universal right coaction of $\underline{\rm aut}^r(A)$ on $A$ is given by $\rho(x_i)=\sum_{j=1}^n x_j\otimes y_{ji}$ for all $i$, with homological codeterminant $g$ \cite[Proposition 5.31(c), Theorem 5.33(b)]{Chirvasitu-Walton-Wang2019}. 

We denote the graded automorphism $\phi$ of $A$ by $\phi(x_i)=\sum_{j=1}^n \gamma_{ji} x_j$ for some coefficients $\gamma_{ji} \in \kk$ such that the matrix $\gamma=(\gamma_{ij})_{1 \leq i,j \leq n}$ is invertible with inverse matrix denoted by $\gamma^{-1} =(\beta_{ij})_{1 \leq i,j \leq n}$. By \cite[Theorem 3.3]{Mori-Smith2016}, we know $(\phi|_V)^{\otimes l}(W)=\hdet(\phi)W$. By choosing $\sum_{1\leq i_1,\ldots,i_l\leq n} f_{i_1\cdots i_l}x_{i_1}\cdots x_{i_l}$ as a basis for $W$, a straightforward calculation yields that
\begin{equation}\label{eq:hdet}
\sum_{1\leq i_1,\ldots,i_l\leq n} f_{i_1\cdots i_l} \gamma_{j_1i_1}\cdots \gamma_{j_li_l} =\hdet(\phi)f_{j_1\cdots j_l}, \quad \textnormal{ for any } 1\leq j_1,\dots, j_l\leq n.
\end{equation}
By the proof of \Cref{lem:2cocycleGrade}, we have a right 2-cocycle $\sigma$ on $\underline{\rm aut}^r(A)$ given by 
\[
\sigma(x,y)=\varepsilon(x)\varepsilon\left(\phi_1^{|x|}(y)\right) \quad \text{and} \quad \sigma^{-1}(x,y)=\varepsilon(x)\varepsilon\left(\phi_2^{|x|}(y)\right)
\]
for two homogeneous elements $x,y\in \underline{\rm aut}^r(A)$, where $(\phi_1,\phi_2)$ is a twisting pair constructed from $\phi$ that satisfies the commutative diagrams \eqref{lem:caut}. By applying \eqref{eq:hdet}, we have two algebra maps $\pi_1,\pi_2: \underline{\rm aut}^r(A)\to \kk$ such that 
\begin{align*}
\pi_1(\mathbb Y)&=\gamma, & \pi_1(\mathbb Z)&=\gamma^{-1}, & \pi_1(g^{\pm 1})&=\left(\hdet(\phi)\right)^{\mp 1},\\
\pi_2(\mathbb Y)&=\gamma^{-1}, & \pi_2(\mathbb Z)&=\gamma, & \pi_2(g^{\pm 1})&=\left(\hdet(\phi)\right)^{\pm 1}.
\end{align*}
It is straightforward to check that $(\phi_1,\phi_2)$ are right and left winding automorphisms of $\pi_1,\pi_2$, respectively (c.f. \cite[Theorem 1.1.7]{HNUVVW21}), such that
\begin{align*}
\phi_1(\mathbb Y)&=\mathbb Y \gamma, & \phi_1(\mathbb Z)&=\mathbb Z \gamma^{-1}, & \phi_1(g^{\pm 1})&=\left(\hdet(\phi)\right)^{\mp 1}g,\\
\phi_2(\mathbb Y)&= \gamma^{-1} \mathbb Y, & \phi_2(\mathbb Z)&= \gamma \mathbb Z, & \phi_2(g^{\pm 1})&=\left(\hdet(\phi)\right)^{\pm 1}g.
\end{align*}
By \Cref{lem:2cocycleGrade}, we know $A^\phi\cong A_\sigma$. Therefore, by \Cref{Cor:Naka}, we can pass the computation of $\mu_{A^\phi}$ to that of $\mu_{A_\sigma}$ as follows.
\begin{align*}
\mu_{A^\phi}(x_i)= \mu_{A_\sigma}(x_i)
&=\sum_{1\leq j\leq n} \widetilde{\alpha_{ji}}x_j\\
&=\sum_{1\leq p,q,r,s,t,j\leq n} \sigma^{-1}\left(y_{jp},S(y_{pq})g\right)\,\sigma^{-1}\left(g^{-1}S^2(y_{st}),g\right)\,\sigma\left(g^{-1}S^2(y_{rs}),S(y_{qr})g\right)\,\alpha_{ti}x_j\\
&=\sum_{1\leq p,q,r,s,t,j\leq n} \varepsilon(y_{jp})\varepsilon\left(\phi_2^{|y_{jp}|}(z_{pq}g)\right) \,\varepsilon\left(g^{-1}S^2(y_{st})\right) \varepsilon\left(\phi_2^{|g^{-1}S^2(y_{st})|}(g)\right) \\
&\qquad \qquad \qquad \quad \varepsilon\left(g^{-1}S^2(y_{rs})\right)\varepsilon\left(\phi_1^{|g^{-1}S^2(y_{rs})|} (z_{qr}g)\right)\,\alpha_{ti}x_j \\
&=\sum_{1\leq p,q,r,s,t,j\leq n}\delta_{jp}\delta_{st}\delta_{rs}\varepsilon\left(\phi_2(z_{pq}g)\right)\,\varepsilon\left(\phi_2^{1+l}(g)\right)\,\varepsilon\left(\phi_1^{1+l}(z_{qr}g)\right)\,\alpha_{ti}x_j\\
&=\sum_{1\leq q,r,j\leq n} \varepsilon\left(\phi_2(g)\right)\,\varepsilon\left((\phi_1\phi_2)^{l+1}(g)\right)\,\varepsilon\left(\phi_2(z_{jq})\right)\, \varepsilon\left(\phi_1^{1+l}(z_{qr})\right)\,\alpha_{ri}x_j\\
&=\sum_{1\leq q,r,t_1,\ldots,t_{l+1},j\leq n} \hdet(\phi)\,\alpha_{jq}\beta_{t_{l+1}r}\cdots \beta_{t_1t_2}\beta_{qt_1}\alpha_{ri}\,x_j\\
&=\sum_{1\leq r,t_2,\ldots,t_{l+1},j\leq n} \hdet(\phi)\,\beta_{t_{l+1}r}\cdots \beta_{jt_2}\alpha_{ri}\,x_j\\
&=(\mu_A\circ \phi^{-l}\circ \xi_{\hdet(\phi)})(x_i).\qedhere
\end{align*}
\end{proof}

\subsection{Finite generation of Hochschild cohomology rings}
\label{subsec:Hochschild}

In this subsection, we discuss the finite generation problem for Hochschild cohomology rings of finite-dimensional algebras under 2-cocycle twists. Throughout, let $H$ be a finite-dimensional commutative Hopf algebra and $A$ be a right $H$-comodule algebra. The \emph{Hochschild cohomology ring} of $A$ is defined to be 
\[{\rm HH}^\bullet(A)~:=~\bigoplus_{n\ge 0}\,\Ext^n_{A\otimes A^{\op}}(A,A),\]
where the $A$-bimodule $A$ is viewed as a left $A\otimes A^{\op}$-module. It is well-known that the Hochschild cohomology ring is a graded-commutative algebra with respect to the cup product. We denote by $\!_A\mathcal M^{H}_A$ the subcategory of ${\rm comod}(H)$ consisting of all $A$-bimodules such that the right and left $A$-module structure maps are $H$-colinear,  together with morphisms between $A$-bimodules being both $H$-colinear and $A$-bilinear. For any right 2-cocycle $\sigma$ on $H$, the monoidal equivalence \eqref{eq:equiv2} between ${\comod}(H)$ and ${\comod}(H^\sigma)$ induces an equivalence on categories
\begin{align*}
\!_A{\mathcal M}_A^H~ &\cong~ \!_{A_\sigma}{\mathcal M}^{H^\sigma}_{A_\sigma}\\
M~&\mapsto ~M_\sigma,
\end{align*}
for any $M\in \!_A{\mathcal M}_A^H$.  

For any two $A$-bimodules $M,N \in \!_A{\mathcal M}_A^H$, we recall the $H$-comodule structure on $\HOM_\kk(M,N)$ by an ungraded analogue of \eqref{eq:coactionHom}. Since $H$ is finite-dimensional, we have $\HOM_\kk(M,N)=\Hom_\kk(M,N)$. And for any $f\in \Hom_\kk(M,N)$, we have $\rho(f)=\sum f_0\otimes f_1\in \Hom_\kk(M,N)\otimes H$ such that
\begin{align*}
\sum f_0(m)\otimes f_1=\sum f(m_0)_0\otimes S^{-1}(m_1)f(m_0)_1,
\end{align*}
for any $m\in M$. We denote by $T=H_\sigma$ the $H$-$H^\sigma$-bi-Galois object, and so $M_\sigma=M\square_HT=(M\otimes T)^{\co H}$, where  $T$ is viewed as a right $H$-comodule via the antipode of $H$ in the last expression. Moreover, we identify $A_\sigma=(A\otimes T)^{\co H}$ with a subalgebra of the tensor product $A\otimes T$ via the right coaction of $T=H_\sigma$ on $A$, $a\mapsto \sum a_0\otimes a_1$, for any $a\in A$. 

The following results are analogues of \Cref{lem:relHom} and \Cref{lem:moduleMTE} for 2-cocycle twists of relative bimodules when $H$ is finite-dimensional and commutative. 

\begin{lemma}
\label{lem:bimodule}
 Let $H$ be a finite-dimensional commutative Hopf algebra and $A$ be a right $H$-comodule algebra. Let $M$ and $N$ be two $A$-bimodules in $\!_{A}\mathcal M^{H}_A$. 
\begin{itemize}
    \item[(1)] $\Hom_{A\otimes A^{\op}}(M,N)$ is a right $H$-subcomodule of $\Hom_\kk(M,N)$.
    \item[(2)] $\Hom_{A\otimes A^{\op}}^H(M,N)=\Hom_{A\otimes A^{\op}}(M,N)^{\co H}$.
    \item[(3)] Suppose $P$ is another $A$-bimodule in $\!_{A}\mathcal M^{H}_A$. Then the natural composition map
    \[
    \Hom_{A\otimes A^{\op}}(N,P)\otimes \Hom_{A\otimes A^{\op}}(M,N)\to \Hom_{A\otimes A^{\op}}(M,P)
    \]
    is $H^{\op}$-colinear.
    \item[(4)] Suppose $R$ is a finite-dimensional right $H$-comodule. Then 
    \[
    \Hom_{A\otimes A^{\op}}(M,N\otimes R)\cong \Hom_{A\otimes A^{\op}}(M,N)\otimes R
    \]
    as right $H$-comodules, where the $A$-bimodule structure on $N\otimes R$ comes from that on $N$.
    \item[(5)]  Let $\sigma$ be a right 2-cocycle on $H$. Then there is a natural isomorphism 
    \[
    \Hom_{A\otimes A^{\op}}(M,N)_\sigma\cong\Hom_{A_\sigma\otimes A_\sigma^{\op}}(M_\sigma,N_\sigma).
    \]
    of $H^\sigma$-comodules.
\end{itemize}
\end{lemma}

\begin{proof}
(1): It suffices to show that, for any $f\in \Hom_\kk(M,N)$ and $m\in M$, the $\kk$-linear map $\rho(f): M\to N\otimes H$ given by $\rho(f)(m)=\sum f(m_0)_0\otimes S^{-1}(m_1)f(m_0)_1$ is $A$-bilinear, where the $A$-bimodule structure on $N\otimes H$ is induced by $N$. By \cite[Lemma 2.1]{CG05}, it is left $A$-linear. For right $A$-linearity, for any $a\in A$, we have 
\begin{align*}
\rho(f)(ma)&=\sum f((ma)_0)_0\otimes S^{-1}((ma)_1)f((ma)_0)_1\\
&=\sum f(m_0a_0)_0\otimes S^{-1}(m_1a_1)f(m_0a_0)_1\\ 
&=\sum f(m_0)_0a_0\otimes S^{-1}(a_2)a_1S^{-1}(m_1)f(m_0)_1\\
&=\sum f(m_0)_0a\otimes S^{-1}(m_1)f(m_0)_1\\
&=(\rho(f)(m))a,
\end{align*}
where the third equality uses the right module structure of $A$ on $N$ and the fact that $H$ is commutative. 

(2): This follows from (1) and \cite[Lemma 1.1]{CG05}.

(3) and (4): The proofs are similar to \Cref{lem:relHom}(3)-(4) by noting that $R$ is finite-dimensional in (4). 

(5): Note that $\Hom_{A\otimes A^{\op}}(M_\sigma,N_\sigma)$ can be realized as an equalizer 
\begin{align*}
\xymatrix{
\Hom_{A_\sigma\otimes A_\sigma^{\op}}(M_\sigma,N_\sigma)\ar[r]&\Hom_{A_\sigma}(M_\sigma,N_\sigma)\ar@<-0.5ex>[rr]_-{f_\sigma\mapsto \triangleleft\circ (f_\sigma\otimes \id_{A_\sigma})}\ar@<0.5ex>[rr]^-{f_\sigma\mapsto f_\sigma\circ \triangleleft} && \Hom_{A_\sigma}(M_\sigma\otimes A_\sigma,N_\sigma),
}
\end{align*}
where the upper and lower symbols $\triangleleft$ denote the actions $M_\sigma\otimes A_\sigma\to M_\sigma$ and $N_\sigma\otimes A_\sigma\to N_\sigma$, respectively. We now apply \cite[Proposition 3.2]{CS2} to identify the right part of the above two arrows with the ones below  
\begin{align*}
\xymatrix{
\Hom_{A}^H(M,N\otimes T)\ar@<-0.5ex>[rr]_-{f\mapsto \triangleleft\circ (f\otimes \id_A)}\ar@<0.5ex>[rr]^-{f\mapsto f\circ \triangleleft} && \Hom_{A}^H(M\otimes A,N\otimes T).
}
\end{align*}
Their equalizer is $\Hom_{A\otimes A^{\op}}^H(M,N\otimes T)$, which yields natural isomorphisms 
\[\Hom_{A_\sigma\otimes A_\sigma^{\op}}(M_\sigma,N_\sigma)\cong \Hom_{A\otimes A^{\op}}^H(M,N\otimes T)\cong (\Hom_{A\otimes A^{\op}}(M,N)\otimes T)^{\co H}\cong \Hom_{A\otimes A^{\op}}(M,N)_\sigma\]
of $H^\sigma$-comodules. 
\end{proof}

We denote by $T^{\op}$ the opposite algebra of $T$. With the same comodule structures as $T$, $T^{\op}$ is a $H$-$H^{\sigma^{\op}}$-bi-Galois object, where $\sigma^{\op}(a,b)=\sigma(b,a)$, for any $a,b\in H$, is a right 2-cocycle on $H$. This implies another monoidal equivalence between ${\comod}(H)$ and ${\comod}(H^{\sigma^{\op}})$ via $(-) ~\square_H ~T^{\op}$, which we denote by $(-)_{\sigma^{\op}}$.

\begin{thm}
\label{thm:FGCHC}
Let $H$ be a finite-dimensional commutative Hopf algebra and $A$ be a right $H$-comodule algebra.  
The Hochschild cohomology ring of $A$ is a graded right $H$-comodule algebra. Moreover, for any right 2-cocycle $\sigma$ on $H$, we have 
    \[{\rm HH}^\bullet(A_\sigma)~\cong~ \left( {\rm HH}^\bullet(A)\right)_{\sigma^{\op}}\] 
as graded right comodule algebras over $H^{\sigma^{\op}}$ for some right 2-cocycle $\sigma^{\op}$ on $H$. As a consequence, if the Hochschild cohomology ring of $A$ is finitely generated, then so is that of $A_\sigma$. 
\end{thm} 
\begin{proof}
Note that $A\in \,_A{\mathcal M}_A^H$. Let 
\[
   P_{\bullet}: \cdots \to P_{2}\to P_1\to P_0\to 0
\]
be the bar resolution for $A$. It is straightforward to check that $P_{\bullet}$ is a resolution for $A$ in $\!_A{\mathcal M}_A^H$ consisting of finite-dimensional free $A$-bimodules. Now we follow the setup before \cite[Lemma 5.9]{KKZ} to use the resolution $P_{\bullet}$ to define a right $K$-coaction on the Hochschild cohomology ring ${\rm HH}^\bullet(A)$. By the construction, we have
\[
{\rm HH}^\bullet(A)=\bigoplus_{n\ge 0} {\rm H}^n\left(\Hom_{A\otimes A^{\op}}(P_{\bullet},A)\right)=\bigoplus_{n\ge 0}{\rm H}^n\left(\Hom_{A\otimes A^{\op}}(P_{\bullet},P_{\bullet})\right).
\]
By definition, the endomorphism algebra $\Hom_{A\otimes A^{\op}}(P_{\bullet},P_{\bullet})$ is a differential graded algebra, where the $\mathbb Z$-grading comes from the complex degree of $P_{\bullet}$. Moreover, the $n$-th term of $\Hom_{A\otimes A^{\op}}(P_{\bullet},P_{\bullet})$ is
\begin{equation}\label{eq:EndHom}
\Hom_{A\otimes A^{\op}(P_{\bullet},P_{\bullet})}^n=\prod_i \Hom_{A\otimes A^{\op}}(P_i,P_{i-n}),
   \end{equation}
where an element in $\Hom_{A\otimes A^{\op}}(P_{\bullet},P_{\bullet})^n$ is written as $(f_i)$ with $f_i: P_i\to P_{i-n}$ for all $i$. The $n$-th differential of $\Hom_{A\otimes A^{\op}}(P_{\bullet},P_{\bullet})$ is $d^n: \Hom_{A\otimes A^{\op}}(P_{\bullet},P_{\bullet})^n\to \Hom_{A\otimes A^{\op}}(P_{\bullet},P_{\bullet})^{n+1}$ defined by
\[
d^n: \left(f_i: P_i\to P_{i-n}\right)\mapsto \left(d_P\circ f_i-(-1)^nf_{i-1}\circ d_P: P_i\to P_{i-n-1}\right),
\]
where $d_P$ is the differential of $P_{\bullet}$. By \Cref{lem:bimodule}(1), each term $\Hom_{A\otimes A^{\op}}(P_i,P_{i-n})$ has a right $H$-comodule structure and hence we have a right $H$-comodule structure on $\Hom_{A\otimes A^{\op}}(P_{\bullet},P_{\bullet})^n$, where we use the fact that $H$ is finite-dimensional. It is straightforward to check the following claims using reasoning similar to that of \cite[Lemma 5.9]{KKZ}.
\begin{itemize}
    \item[(a)] $\Hom_{A\otimes A^{\op}}(P_{\bullet},P_{\bullet})$ is a right $H^{\op}$-comodule algebra.
    \item[(b)] The right $H^{\op}$-coaction on $\Hom_{A\otimes A^{\op}}(P_{\bullet},P_{\bullet})$ preserves the $\mathbb Z$-grading.
    \item[(c)] The differential $d^n$ is $H^{\op}$-colinear. 
\end{itemize}
Indeed, part (a) is from \Cref{lem:bimodule}(3). Part (b) can be derived from the definition of $H$-comodule structure on $\Hom_{A\otimes A^{\op}}(P_{\bullet},P_{\bullet})$ via \eqref{eq:coactionHom}. Part (c) follows from the fact that $d_P$ is $H$-colinear. Since the Hochschild cohomology ring ${\rm HH}^\bullet(A)$ can be computed as the cohomology ring of the differential graded $H$-comodule algebra $\Hom_{A\otimes A^{\op}}(P_{\bullet},P_{\bullet})$, our first result follows immediately. 

Next, we note that 
\[
P_{\bullet \sigma}:  \cdots\to  (P_2)_\sigma\to (P_1)_\sigma\to (P_0)_\sigma\to 0
\]
is a resolution of $A_\sigma$ in $\!_{A_\sigma}{\mathcal M}_{A_\sigma}^{H^\sigma}$, where each term is a free $A_\sigma$-bimodule. In particular, $P_{\bullet \sigma}$ is a projective resolution of the $A_\sigma$-bimodule $A_\sigma$. By a similar discussion as before, we can also view $\Hom_{A_\sigma\otimes A_\sigma^{\op}}(P_{\bullet \sigma},P_{\bullet \sigma})$ as a differential $\mathbb Z$-graded algebra in the monoidal category ${\rm comod}((H^\sigma)^{\op})$. It is easy to check that $(H^\sigma)^{\op}
\cong (H^{\op})^{\sigma^{\op}}=H^{\sigma^{\op}}$. Therefore, by \Cref{lem:bimodule}, we have the following isomorphisms
\begin{align*}
{\rm HH}^\bullet(A_\sigma)&={\rm H}^\bullet\left(\Hom_{A_\sigma\otimes A_\sigma^{\op}}(P_{\bullet \sigma},P_{\bullet \sigma})\right)\\
&\cong{\rm H}^\bullet\left(\Hom_{A\otimes A^{\op}}(P_{\bullet},P_{\bullet})_{\sigma^{\op}}\right)\\
&\cong{\rm H}^\bullet\left(\Hom_{A\otimes A^{\op}}(P_{\bullet},P_{\bullet})\right)_{\sigma^{\op}}\\
&\cong {\rm HH}^\bullet(A)_{\sigma^{\op}}
\end{align*}
of $\mathbb Z$-graded right $H^{\sigma^{\op}}$-comodule algebras. The last assertion is a consequence of \cite[Proposition 3.1(1)]{M2005} applied on ${\rm HH}^\bullet(A)_{\sigma^{\op}}$. 
\end{proof}

\section{The quantum-symmetric equivalence class of a Koszul AS-regular algebra}

To conclude our paper, we provide an answer to \Cref{QuestionB} in the introduction for a  Koszul AS-regular algebra (without the requirement of a finite GK-dimension). We study its quantum-symmetric equivalence class by combining the results of Raedschelders and Van den Bergh with ours on homological Morita--Takeuchi invariants. We show that all Koszul AS-regular algebras of a fixed global dimension belong to a single quantum-symmetric equivalence class.

\subsection{Determining the quantum-symmetric equivalence class}

In \cite[Theorem 7.2.3]{vdb2017}, Raedschelders and Van den Bergh proved the Morita--Takeuchi equivalence between Manin’s universal quantum groups associated with Koszul AS-regular algebras of the same dimension (in fact, their theorem holds more generally without any assumption of finite GK-dimension on the algebra). Using their result, we are now able to deduce the following.

\begin{thm}
\label{thm:TwistAS}
Suppose $\kk$ is algebraically closed. Let $A$ be a  Koszul AS-regular algebra, and $B$ be any connected graded algebra generated in degree one. Then $A$ and $B$ are quantum-symmetrically equivalent if and only if $B$ is a Koszul AS-regular algebra of the same global dimension as $A$. In this case, $A$ and $B$ are 2-cocycle twists of one another if and only if they have the same Hilbert series.
\end{thm}

\begin{proof}
 We follow the discussion in \cite[\S 15]{Manin2018}. For any integer $d\ge 1$, consider the rigid monoidal category $\mathcal U$ constructed in \cite[\S 15.5.1]{Manin2018}, whose objects are given by the free monoid 
\[
\Lambda=\langle r_1,\ldots,r_{d-1},r_d^{\pm 1}\rangle 
\]
and morphisms are generated by
\[
\phi_{a,b}: r_{a+b}\to r_a\otimes r_b\quad \text{and}\quad \theta_{a,b}: r_a\otimes r_d^{-1}\otimes r_b\to r_{a+b-d},
\]
satisfying certain relations. Let $A=TV/(R)$ be a Koszul, AS-regular algebra of global dimension $d$. Set $R_1=V$ and  $R_k=\bigcap_{i+j+2=k} V^{\otimes i}\otimes R\otimes V^{\otimes j}$ for $k\ge 2$, which are all finite-dimensional right comodules over Manin's universal quantum group $\uaut^r(A)$. In particular, $R_d$ is a one-dimensional invertible comodule. By \cite[\S 15.5.2]{Manin2018}, there is a monoidal functor $M_A: \mathcal U\to {\rm Vect_{fd}} (\kk)$ such that $M_A(r_i)=R_i$ and the morphisms $M_A(\phi_{a,b}): R_{a+b}\to R_a\otimes R_b$ and $M_A({\theta_{a,b}}): R_a\otimes R_d^{-1}\otimes R_b\to R_{a+b-d}$ are natural comodule morphisms over $\uaut^r(A)$. In particular, one notes that 
\[A=T\langle R_1\rangle/(R_2)=T\langle M_A(r_1)\rangle/(M_A(r_2))\]
 is a right graded comodule algebra over $\uaut^r(A)$ via its universal coaction. 
 
 Suppose $B$ is another Koszul AS-regular algebra of global dimension $d$ and let $M_B$ be the corresponding functor defined analogously to $M_A$. According to \cite[Theorem 15.44]{Manin2018}, we can construct a monoidal equivalence between ${\rm comod}\left(\uaut^r(A)\right)$ and ${\rm comod}\left(\uaut^r(B)\right)$ by using the rigid monoidal category $\mathcal U$ and the two corresponding monoidal functors $M_A, M_B: \mathcal U\to {\rm Vect_{fd}} (\kk)$ as follows.  Denote
 $\kk \mathcal U$ as the $\kk$-linear rigid monoidal category by linearizing the morphism spaces in the rigid monoidal category $\mathcal U$, where we still write $M_A: \kk \mathcal U \to {\rm comod}\left(\uaut^r(A)\right)$ as the linearized functor of $M_A$. According to \cite[Theorem 15.42]{Manin2018}, $M_A: \kk \mathcal U \to {\rm comod}\left(\uaut^r(A)\right)$ is fully faithful. Now, let ${\rm Perf}(\mathcal U^{\op})$ be the triangulated category of finite
complexes of finitely generated projective right $\mathcal U$-modules. Note that ${\rm Perf}(\mathcal U^{\op})$ is 
 a triangulated monoidal category by extending the tensor product $\kk\mathcal U(-, r_i) \otimes \kk\mathcal U(-, r_j)=\kk\mathcal U(-, r_i\otimes r_j)$ to complexes. Hence the functor $M_A: \kk \mathcal U \to {\rm comod}\left(\uaut^r(A)\right)$ extends to an exact monoidal functor 
 \[
M_A: {\rm Perf}(\mathcal U^{\op}) \to D^b({\rm comod}\left(\uaut^r(A)\right)),
\]
 which we still denote by $M_A$. A result of \cite[Theorem 15.42]{Manin2018} shows that the above functor is indeed an equivalence of monoidal triangulated categories. One then can use the above derived equivalence to translate the standard $t$-structure on the derived category $D^b({\rm comod}\left(\uaut^r(A)\right)$ to that on ${\rm Perf}(\mathcal U^{\op})$ (for background on derived categories and $t$-structures, see e.g., \cite{Yekutieli}). This gives a monoidal equivalence $G_A$ between an intrinsic monoidal full subcategory $\mathcal C$ of ${\rm Perf}(\mathcal U^{\op})$ and ${\rm comod}\left(\uaut^r(A)\right)$ (see the description of $\mathcal C$ in the proof of \cite[Theorem 7.2.3]{vdb2017}). We write the monoidal inverse equivalence $G_A^{-1}: {\rm comod}\left(\uaut^r(A)\right) \xrightarrow{\sim} \mathcal C$. Applying the same argument to ${\rm comod}\left(\uaut^r(B)\right)$ with respect to $M_B$, we get another monoidal equivalence $G_B: \mathcal C\xrightarrow{\sim} {\rm comod}\left(\uaut^r(B)\right)$. As a consequence, there is a monoidal equivalence 
 \[F: {\rm comod}\left(\uaut^r(A)\right)\xrightarrow{G_A^{-1}}\mathcal C\xrightarrow{G_B}{\rm comod}\left(\uaut^r(B)\right).\] 
 By its construction,  $F$ sends $M_A(r_i)$ to $M_B(r_i)$ for $1\leq i\leq d$. In particular, 
\[
F(A)=F\left(T\langle M_A(r_1)\rangle/(M_A(r_2)\right)=T\langle M_B(r_1)\rangle/(M_B(r_2))=B
\]
as graded comodule algebras over $\uaut^r(B)$. Thus, $A$ and $B$ are quantum-symmetrically equivalent. 

For the other direction, suppose $A$ and $B$ are quantum-symmetrically equivalent. The previous discussion shows that $A$ is quantum-symmetrically equivalent to the polynomial ring $R=\kk[x_1,\ldots, x_d]$ since $R$ is  Koszul AS-regular of global dimension $d$. Therefore, $B$ is quantum-symmetrically equivalent to $R$, since quantum-symmetrical equivalence is an equivalence relation. By \Cref{Cor:gldim}(2), $B$ is Koszul. Moreover since $R$ is noetherian and $\uaut(R)$ has bijective antipode \cite[Corollary 5.3]{vdb2017}, \Cref{thm:Torreg}(5) implies that $B$ is AS-regular of global dimension $d$. 

Now assume that $A$ and $B$ are two Koszul AS-regular algebras. If they have the same Hilbert series, then we have 
\begin{equation}\label{eq:dimeq}
   \dim_\kk M_A(r_i)=\dim_\kk (A^!)_i=\dim_\kk (B^!)_i=\dim_\kk M_B(r_i) 
\end{equation} 
for all $1\leq i\leq d$. By \cite[Corollary 15.43]{Manin2018}, we know that the Grothendieck ring of $\uaut^r(A)$  is the free $\mathbb Z$-ring $\mathbb Z\langle M_A(r_1),\ldots, M_A(r_{d-1}), M_A(r_d)^{\pm 1}\rangle$, and the same is true for $\uaut^r(B)$. Then the monoidal equivalence $F: {\rm comod}\left(\uaut^r(A)\right)\xrightarrow{\sim}{\rm comod}\left(\uaut^r(B)\right)$ induces a ring isomorphism $F$ (by abuse of notation) between the Grothendieck rings of $\uaut^r(A)$ and $\uaut^r(B)$ by sending $M_A(r_i)$ to $M_B(r_i)$ for all $1\leq i\leq d$. For any finite-dimensional comodule $V$ over $\uaut^r(A)$, we can consider $V$ as an element in the Grothendieck ring of $\uaut^r(A)$ and write it as a $\mathbb Z$-linear combination of some tensor products of the free generators  $M_A(r_1),\ldots,M_A(r_d)^{\pm 1}$. By applying $F$, one can similarly write $F(V)$ as the $\mathbb Z$-linear combination of the tensor products now in terms of $M_B(r_1),\ldots,M_B(r_d)^{\pm 1}$ in the Grothendieck ring of $\uaut^r(B)$.  This implies that $\dim_\kk F(V)=\dim_\kk V$ due to \eqref{eq:dimeq} and hence $F$ preserves dimensions. As a consequence, $\uaut^r(B)$ is a 2-cocycle twist of $\uaut^r(A)$ (see \cite[Theorems 1.8 and 1.17]{Bichon2014} and \cite[Proposition 4.2.2]{PavelGelaki2001}), where $F$ is naturally isomorphic to the monoidal equivalence given by the 2-cocycle twist. Since $F$ maps the comodule algebra $A$ over $\uaut^r(A)$ to the comodule algebra $B$ over $\uaut^r(B)$, we know $B$ is a 2-cocycle twist of $A$. Conversely, if $A$ and $B$ are 2-cocycle twists of each other, they share the same Hilbert series by construction.
\end{proof}

Let $A$ be an AS-regular algebra of global dimension $d$. By \cite[Theorem 5.11]{Smith1994}, $A$ is Koszul if it has the Hilbert series of the form $1/(1-t)^d$. Conversely, if $A$ is Koszul and has finite GK-dimension, it is unknown whether the Hilbert series of $A$ is of the form $1/(1-t)^d$ (see e.g., the survey paper \cite{Rogalski2022}). If this were true, then \Cref{thm:TwistAS} would imply that all Koszul AS-regular algebras of global dimension $d$ and finite GK-dimension are 2-cocycle twists of each other.

The following corollary is straightforward. 

\begin{Cor}\label{Cor:qs-equivalence}
The quantum-symmetric equivalence class of the polynomial algebra $\kk[x_1,\ldots,x_d]$ consists of all Koszul AS-regular algebras of dimension $d$.    
\end{Cor}

\subsection{The 4-dimensional Sklyanin algebras and a conjecture}

By \cite[Theorem 0.3]{SSta92}, when $\{\alpha,\beta,\gamma\}$ is not equal to the triples $\{-1,1,\gamma\}, \{\alpha,-1,1\}$, or $\{1,\beta,-1\}$, the family of 4-dimensional Sklyanin algebras $S(\alpha,\beta,\gamma)$ is Koszul AS-regular with the same Hilbert series as that of the commutative polynomial ring $\kk[x_0,x_1,x_2,x_3]$ in 4 variables. Hence, by \Cref{thm:TwistAS}, $S(\alpha,\beta,\gamma)$ is a 2-cocycle twist of $\kk[x_0,x_1,x_2,x_3]$, which we give a conjectural description of as follows. 

Let $V$ be a 4-dimensional vector space over $\kk$ with basis $\{x_0,x_1,x_2,x_3\}$. Then the polynomial ring $\kk[x_0,x_1,x_2,x_3]$ is the superpotential algebra $A(e,2)$ with the superpotential 
\[
e=\sum_{\sigma\in S_4} {\rm sgn}(\sigma)x_{\sigma(0)}\,x_{\sigma(1)}x_{\sigma(2)}x_{\sigma(3)}\in V^{\otimes 4}.
\]
One may explicitly describe $S(\alpha,\beta,\gamma)$ as a superpotential algebra $A(f,2)$ for some superpotential $f\in V^{\otimes 4}$ by \cite[Lemma 8.5]{Chirvasitu-Walton-Wang2019}. Applying the cogroupoid $\mathcal{GL}_4$ associated to $4$-preregular forms constructed in \cite[Definition 3.1.1]{HNUVVW2}, we obtain the Hopf algebra $\mathcal{GL}_4(e,e)$ that is isomorphic to Manin's universal quantum group $H=\underline{\rm aut}^l(\kk[x_0,x_1,x_2,x_3])$, and likewise $\mathcal{GL}_4(f,f)$ is isomorphic to $K=\uaut^l(S(\alpha,\beta,\gamma))$. Moreover, there is a $K$-$H$-bicomodule algebra $T$ (denoted by $\mathcal{GL}_4(f,e)$ in \cite{HNUVVW2}) given by the $\kk$-algebra with $34$ generators
    \[\mathbb A=(a_{ij})_{1\leq i,j\leq 4}, \qquad \mathbb B=(b_{ij})_{1\leq i,j\leq 4}, \qquad D^{\pm 1},\]
subject to the relations 
\begin{equation*}
\left. 
\begin{aligned}
    \sum_{1\leq i_1,i_2,i_3,i_4\leq 4}{f}_{i_1 i_2i_3i_4}a_{i_1j_1}a_{i_2j_2}a_{i_3j_3} a_{i_4j_4} &= {e}_{j_1j_2j_3j_4}D, &\textnormal{ for any } 1\leq j_1,j_2,j_3, j_4\leq 4,\\
    \sum_{1\leq i_1,i_2,i_3,i_4\leq 4}{e}_{i_1i_2i_3i_4}b_{i_4j_4}b_{i_3j_3}b_{i_2j_2} b_{i_1j_1} &= {f}_{j_1j_2j_3j_4}D^{-1},  &\textnormal{ for any } 1\leq j_1,j_2,j_3,j_4\leq 4,\\
    DD^{-1} &= D^{-1}D=1, \text{ and } \\
    \mathbb{A}\mathbb{B} &= \mathbb{I}_{4 \times 4}.
\end{aligned}\right\}
\end{equation*}
We have the following expectation: 
\begin{conj}
\label{conj-bigalois}
The above bicomodule algebra $T$ is a nonzero cleft bi-Galois object, consequently yielding a 2-cocycle which twists $H$ to $K$.
\end{conj}
If \Cref{conj-bigalois} holds, take any right $H$-comodule isomorphism 
$\phi: H\to T$
that is convolution invertible with inverse denoted by $\phi^{-1}: H \to T$. This yields a left 2-cocycle $\sigma$ on $H$ given by
\[
\sigma(x,y)=\sum \phi(x_1)\phi(y_1)\phi^{-1}(x_2y_2)\in \kk
\]
for any $x,y\in H$. As a consequence, we obtain an algebra isomorphism
\[\!_\sigma\left(\kk[x_0,x_1,x_2,x_3]\right)\cong S(\alpha,\beta,\gamma).\]

\bibliography{biblNov2021}
\bibliographystyle{amsplain}
\end{document}